\documentclass[a4,11pt]{article}
\usepackage[subrefformat=parens]{subcaption}
\usepackage[dvips]{graphicx}
\usepackage{amssymb}
\usepackage{amstext}
\usepackage{amsmath}
\usepackage{amscd}
\usepackage{amsthm}
\usepackage{amsfonts}
\usepackage{enumerate}
\usepackage{graphicx}
\usepackage{latexsym}
\usepackage{mathrsfs}
\usepackage{mathtools}
\usepackage{cases}
\usepackage[svgnames]{xcolor}
\usepackage{verbatim}
\usepackage{bm}
\usepackage{tikz}
\usepackage{caption}
\usepackage[top=3.5cm, bottom=3.5cm, left=3.5cm, right=3.5cm]{geometry}
\usepackage[labelformat=empty,labelsep=none]{subcaption}
\usepackage{authblk}
\usepackage{comment}
\usepackage{hyperref}
\newtheorem{thm}{Theorem}[section]

\newtheorem{lem}[thm]{Lemma}

\theoremstyle{definition}
\newtheorem{rem}[thm]{Remark}

\newtheorem*{claim*}{Claim}
\theoremstyle{remark}

\numberwithin{equation}{section}

\title{\vspace{-3cm}\textbf{Inverse medium scattering problems with Kalman filter techniques}}

\author[1]{\rm Takashi Furuya}
\author[2]{\rm Roland Potthast}
\affil[1]{ {\small Department of Mathematics, Hokkaido University, Japan}}
\affil[ ]{{\small Email: takashi.furuya0101@gmail.com}\vspace{3mm}}

\affil[2]{{\small Data Assimilation Unit, Deutscher Wetterdienst, Germany}}
\affil[ ]{{\small Email: Roland.Potthast@dwd.de}}

\date{}
\begin{document}
\maketitle
\begin{abstract}
We study the inverse medium scattering problem to reconstruct the unknown inhomogeneous medium from the far-field patterns of scattered waves. 
The inverse scattering problem is generally ill-posed and nonlinear, and the iterative optimization method is often adapted.
A natural iterative approach to this problem is to place
all available measurements and mappings into one long vector and mapping, respectively, and to iteratively solve the linearized large system equation using the Tikhonov regularization method, which is called the Levenberg-Marquardt scheme.
However, this is computationally expensive because we must construct the larger system equations when the number of available measurements increases.
In this paper, we propose two reconstruction algorithms based on the Kalman filter.
One is the algorithm equivalent to the Levenberg–Marquardt scheme, and the other is inspired by the Extended Kalman Filter.
For the algorithm derivation, we iteratively apply the Kalman filter to the linearized equation for our nonlinear equation.
Our proposed algorithms sequentially update the state and the weight of the norm for the state space, which avoids the construction of a large system equation and retains the information of past updates.
Finally, we provide numerical examples to demonstrate our proposed algorithms.
\end{abstract}
\date{{\bf Key words}.
Inverse acoustic scattering, Inhomogeneous medium, Far-field pattern, Tikhonov regularization method, Levenberg–Marquardt, Kalman filter, Extended Kalman filter.}
\section{Introduction}
The inverse scattering problem is used to determine unknown scatterers by measuring the scattered waves generated by sending incident waves far away from the scatterers.
This is important in many applications, such as medical imaging, non-destructive testing, remote exploration, and geophysical prospecting. 
Owing to its many applications, the inverse scattering problem has been studied from various perspectives.
For further reading, we refer to the following books \cite{Cakoni, Chen, ColtonKress, Kirsch, NakamuraPotthast}, which include reviews of classical and recent developments of the inverse scattering problem.
\par
We begin with the mathematical formulation of the scattering problem.
Let $k>0$ be the wave number and $\theta \in \mathbb{S}^{1}$ be the incident direction. 
We denote the incident field $u^{inc}(\cdot, \theta)$ with direction $\theta$ by the plane wave of the form 
\begin{equation}
u^{inc}(x, \theta):=\mathrm{e}^{ikx \cdot \theta}, \ x \in \mathbb{R}^2. 
\end{equation}
Let $Q \subset \mathbb{R}^2$ be a bounded open set with the smooth boundary and its exterior $\mathbb{R}^2\setminus  \overline{Q}$ be connected.
We assume that $q \in L^{\infty}(\mathbb{R}^2)$, which refers to the inhomogeneous medium, satisfies $\mathrm{Re}(1+q)>0$, $\mathrm{Im}q \geq 0$, and its support $\mathrm{supp}\ q$ is embedded in $Q$, that is, $\mathrm{supp}\ q \Subset Q$. 
Then, the direct scattering problem concerns determining the total field $u=u^{sca}+u^{inc}$ such that
\begin{equation}
\Delta u+k^2(1+q)u=0 \ \mathrm{in} \ \mathbb{R}^2, \label{1.2}
\end{equation}
\begin{equation}
\lim_{r \to \infty} \sqrt{r} \biggl( \frac{\partial u^{sca}}{\partial  r}-iku^{sca} \biggr)=0, \label{1.3}
\end{equation}
where $r=|x|$. 
The {\it Sommerfeld radiation condition} (\ref{1.3}) holds uniformly in all directions $\hat{x}:=\frac{x}{|x|}$. 
Furthermore, the problem (\ref{1.2})--(\ref{1.3}) is equivalent to the {\it Lippmann-Schwinger integral equation}
\begin{equation}
u(x, \theta)=u^{inc}(x, \theta)+k^2\int_{Q}q(y)u(y, \theta)\Phi(x,y)dy, \label{1.4}
\end{equation}
where $\Phi(x,y)$ denotes the fundamental solution of the Helmholtz equation in $\mathbb{R}^2$, that is, 
\begin{equation}
\Phi(x,y):= \displaystyle \frac{i}{4}H^{(1)}_0(k|x-y|), \ x \neq y, \label{1.5}
\end{equation}
where $H^{(1)}_0$ is the Hankel function of the first kind of order one. 
It is well known that there exists a unique solution $u^{sca}$ of the problem (\ref{1.2})--(\ref{1.3}), and it has the following asymptotic behaviour,
\begin{equation}
u^{sca}(x, \theta)=\frac{\mathrm{e}^{ikr}}{\sqrt{r}}\Bigl\{ u^{\infty}(\hat{x},\theta)+O\bigl(1/r \bigr) \Bigr\} , \ r \to \infty, \ \ \hat{x}:=\frac{x}{|x|}. \label{1.6}
\end{equation}
The function $u^{\infty}$ is called the {\it far field pattern} of $u^{sca}$, and has the form
\begin{equation}
u^{\infty}(\hat{x},\theta)
=
\frac{k^2\mathrm{e}^{\frac{i\pi}{4}} }{\sqrt{8\pi k}}
\int_{Q}
\mathrm{e}^{-ik \hat{x} \cdot y} u(y, \theta)q(y)dy=:\mathcal{F}_{\theta}q(\hat{x}), \label{1.7}
\end{equation}
where the far-field mapping $\mathcal{F}_{\theta}:L^{2}(Q) \to L^{2}(\mathbb{S}^{1})$ is defined in the second equality for each incident direction $\theta \in \mathbb{S}^{1}$. 
For more details on these direct scattering problems, see Chapter 8 of \cite{ColtonKress}. 
\par
We consider the inverse scattering problem to reconstruct function $q$ from the far field pattern $u^{\infty}(\hat{x}, \theta_n)$ for all directions $ \hat{x} \in \mathbb{S}^{1}$ and several directions $\{ \theta_n \}_{n=1}^{N}\subset \mathbb{S}^{1}$ with some large $N \in \mathbb{N}$, and one fixed wave number $k>0$.
It is well known that function $q$ is uniquely determined from the far field pattern $u^{\infty}(\hat{x}, \theta)$ for all $ \hat{x}, \theta \in \mathbb{S}^{1}$ and one fixed $k>0$ (see, e.g., \cite{bukhgeim2008recovering, novikov1988multidimensional, ramm1988recovery}).
However, the uniqueness for several incident plane waves is an open question.
For the impenetrable obstacle scattering case, if we assume that the shape of the scatterer is a polyhedron or ball, the uniqueness for a single incident plane wave is proved (see \cite{alessandrini2005determining, cheng2003uniqueness, liu1997inverse, liu2006uniqueness}).
Recently, in \cite{alberti2020infinitedimensional}, they showed the Lipschitz stability for the inverse medium scattering with finite measurements $\{ u^{\infty}(\hat{x}_{i}, \theta_{j}) \}_{i,j=1,...,N}$ for large $N \in \mathbb{N}$ under the assumption that the true function belongs to a compact and convex subset of finite-dimensional subspace. 
From these results, we expect that a large number of measurements $N$ is necessary to reconstruct the general shape of the scatterer.
\par
Our inverse problem (\ref{1.7}) is not only ill-posed, but also nonlinear. 
Existing methods for solving the nonlinear inverse problem can be roughly categorized into two groups: iterative optimization  and qualitative.
Iterative optimization methods (see, e.g., \cite{Bakushinsky, ColtonKress, Giorgi, Hohage, Kaltenbacher}) do not require many measurements, but an initial guess, which is the starting point of the iteration, is required.  
It must be appropriately chosen by a priori knowledge of the unknown function $q$, otherwise, the iterative solution cannot converge to the true function. 
Moreover, the qualitative methods such as the linear sampling method \cite{ColtonKirsch}, no-response test \cite{Honda}, probe method \cite{Ikehata}, factorization method \cite{KirschGrinberg}, and singular source method \cite{Potthast}, do not require an initial guess, and they are computationally faster than iterative methods.
However, the disadvantage of the qualitative method is that it requires an uncountable number of measurements. 
For a survey of qualitative methods, we refer to \cite{NakamuraPotthast}.
Recently, in \cite{ito2012direct, Liu_2018}, they suggested a reconstruction method from a single incident plane wave, although rigorous justifications are lacking.
Another approach is the {\it Born approximation}, which approximates the total field $u$ in (\ref{1.7}) with the incident field $u^{inc}$, and then the nonlinear equation is transformed into a linear equation.
Such an approximation is appropriate when the wavenumber $k>0$ and value of $q$ are small (see the second term in the right-hand side of (\ref{1.4})).
For further readings on the Born approximation, see \cite{Bakushinsky, Bao, ColtonKress, Kirsch2, Pike}. 
Nowadays, data-driven approaches are being developed; we refer to \cite{arridge_maass_oktem_schonlieb_2019} for their survey.
\par
Here, we focus on the iterative optimization method.
It is basically based on the {\it Newton method} (see, e.g., \cite{Bakushinsky, ColtonKress, Kaltenbacher, Kirsch, Kress, NakamuraPotthast}), which is a classical method for constructing an iterative solution by first-order linearization.
A natural iterative approach to our problem is to place all available measurements $\{ u^{\infty}(\cdot, \theta_n) \}_{n=1}^{N}$ and far field mappings $\{ \mathcal{F}_{\theta_{n}} \}_{n=1}^{N}$ into one large vector $\vec{u}^{\infty}$ and mapping $\vec{\mathcal{F}}$, respectively, and to iterative solve the linearized large system for $\vec{u}^{\infty}=\vec{\mathcal{F}}q$ by the {\it Tikhonov regularization method} (see, e.g., \cite{Cakoni, Hanke, Kress, NakamuraPotthast}), which is called {\it Levenberg–Marquardt} scheme (see, e.g., \cite{Hanke_1997, Kaltenbacher}).
However, this is computationally expensive because we must construct the larger vector $\vec{u}^{\infty}$ and mapping $\vec{\mathcal{F}}$ when the number of measurements $N$ is increasing.
\par
In this paper, we propose two reconstruction algorithms based on the {\it Kalman filter} without using  $\vec{u}^{\infty}$ and $\vec{\mathcal{F}}$.
The contributions of this paper are the following.
\begin{itemize}
\item[(A)] We propose a reconstruction algorithm equivalent to the Levenberg-Marquardt scheme (see (\ref{KFL-1})--(\ref{KFL-3}) and Theorem \ref{equivalence for KFN and FTN}). 
\item[(B)] We propose a reconstruction algorithm inspired by Extended Kalman filter (see (\ref{EKF-1})--(\ref{EKF-3})). 
\end{itemize}
The Kalman filter (see, e.g., \cite{Freitag, Kalman, NakamuraPotthast}) is an algorithm that estimates the unknown state in a linear system. 
It sequentially updates the state and the weight of the norm for the state space, which avoids constructing the large system equation and retains the information of past updates.
For the algorithm derivation, we iteratively apply the Kalman filter to the linearized equation for our nonlinear equation (\ref{1.7}).
We call algorithm (A) the {\it Kalman filter Levenberg–Marquardt} (KFL) and algorithm (B) the {\it iterative Extended Kalman Filter} (EKF). 
KFL is based on the linearization at the initial state for each iteration step, whereas EKF is based on the linearization at the current state for every iteration step, implying that the update of KFL is slower than that of EKF (see Remark \ref{EKF and KFL}). 
For both algorithms, we can select the initialization or update of the weight of the norm for each iteration step depending on the situation (see Remark \ref{update B}), which is a beneficial result from viewing the iterative algorithm as a Kalman filter.
We provide numerical examples to demonstrate the proposed algorithms and observed that, in both the noise-free and noise-free cases, the error of EKF decreases faster than that of KFL if the regularization parameter is chosen appropriately (see Figures \ref{B1}--\ref{B2noise}).
However, EKF is more sensitive to the regularization parameter and noise than KFL (see Figures \ref{Graph sigma} and \ref{Graph alpha}).
We also observed that, in both algorithms, if we select the update of the weight in (\ref{update weight B}), the algorithms become more robust to the regularization parameter and noise, although the error slowly decreases (see Figures \ref{Graph sigma} and \ref{Graph alpha}).
\par
The remainder of this paper is organized as follows. 
In Section \ref{Preliminary}, we briefly review the far-field mapping and the Kalman filter.
In Section \ref{Kalman filter Levenberg–Marquardt}, we propose the KFL reconstruction algorithm.
In Section \ref{Iterative Extended Kalman filter}, we propose the EKF reconstruction algorithm.
Finally, in Section \ref{Numerical examples}, we provide numerical examples to demonstrate the algorithms.
\section{Preliminary}\label{Preliminary}
\subsection{Far field mapping}
In this section, we briefly review the Fr'echet derivative of the far-field mapping, the Lipschitz stability, and the Levenberg-Marquardt scheme for inverse scattering.
We redefine the far filed mapping $\mathcal{F}_{\theta}: L^{2}(Q) \to L^{2}(\mathbb{S}^{1})$ by 
\begin{equation}
\mathcal{F}_{\theta}q(\hat{x})
:=
\frac{k^2\mathrm{e}^{\frac{i\pi}{4}} }{\sqrt{8\pi k}}
\int_{Q}\mathrm{e}^{-ik \hat{x} \cdot y} u_{q}(y, \theta)q(y)dy, \ \hat{x} \in \mathbb{S}^1, \label{2.1}
\end{equation}
where the total field $u_{q}(\cdot, \theta)$ is given by solving the integral equation of (\ref{1.4}).
In addition, We denote $L^{\infty}_{+}(Q):=\{q \in L^{\infty}(Q): \exists q_{0}>0,\ \mathrm{Im}q \geq q_{0} \ \mathrm{a.e} \ \mathrm{on} \ Q \} \subset L^{2}(Q)$.
The following lemma is proved by the same argument as in Section 2 of \cite{Bao2005InverseMS}.

\begin{lem}\label{Lemma 2.1}
\begin{description}
\item[(1)] $\mathcal{F}_{\theta} \in C^{1}(L^{\infty}_{+}(Q), L^{2}(\mathbb{S}^{1}))$, that is, for any $q \in L^{\infty}_{+}(Q)$, $\mathcal{F}_{\theta}$ is Fr\'echet differentiable at $q$, and by denoting the
Fr\'echet derivative by $\mathcal{F}^{\prime}_{\theta}[q]:L^{2}(Q) \to L^{2}(\mathbb{S}^{1})$, the mapping $q \in L^{\infty}_{+}(Q) \mapsto \mathcal{F}^{\prime}_{\theta}[q] \in \mathcal{L}(L^{2}(Q), L^{2}(\mathbb{S}^{1}))$ is continuous, and its derivative $\mathcal{F}^{\prime}_{\theta}[q]$ at $q$ is given by
\begin{equation}
\mathcal{F}^{\prime}_{\theta}[q]m=v^{\infty}_{q,m}, \label{2.2}
\end{equation} 
where $v^{\infty}_{q,m}$ is the far-field pattern of the radiating solution $v=v_{q,m}$ such that
\begin{equation}
\Delta v+k^2(1+q)v=-k^{2}mu_{q}(\cdot, \theta) \ \mathrm{in} \ \mathbb{R}^2. \label{2.3}
\end{equation}
\item[(2)] $\mathcal{F}^{\prime}_{\theta}[\cdot]$ is locally bounded.
\end{description}
\end{lem}

We denote the far field mappings for all incident directions $\mathcal{F}: L^{\infty}_{+}(Q) \subset L^{2}(Q) \to L^{2}(\mathbb{S}^{1}\times \mathbb{S}^{1})$ by $\mathcal{F}q(\hat{x}, \theta):=\mathcal{F}_{\theta}q(\hat{x})$.
The following lemma is proved by the same argument as in Section 11 of \cite{ColtonKress} and Section 2 of \cite{Bao2005InverseMS}.

\begin{lem}\label{prop of Frechet}
\begin{description}
\item[(1)] $\mathcal{F}: \{q \in L^{\infty}(Q): \mathrm{Im}q \geq 0 \ \mathrm{a.e.} \ \mathrm{on} \ Q \} \to L^{2}(\mathbb{S}^{1} \times \mathbb{S}^{1})$ is injective.
\item[(2)] $\mathcal{F}\in C^{1}(L^{\infty}_{+}(Q), L^{2}(\mathbb{S}^{1} \times \mathbb{S}^{1}))$, and its derivative  $\mathcal{F}^{\prime}[q]: L^{2}(Q) \to L^{2}(\mathbb{S}^{1} \times \mathbb{S}^{1})$ at $q$ is injective.
\end{description}
\end{lem}

By applying Theorem 2.1 of \cite{BOURGEOIS2013187} with Lemma \ref{prop of Frechet}, we obtain the following Lipschitz stability.

\begin{lem}\label{Lemma 2.3}
Let $W$ be a finite-dimensional
subspace of $L^{2}(Q)$ and $K$ be a compact and convex subset of $W \cap L^{\infty}_{+}(Q)$. Then, there exists a constant $C>0$ such that
\begin{equation}
\left\| q_{1} - q_{2} \right\|_{L^{2}(Q)} \leq C \left\| \mathcal{F}q_{1} - \mathcal{F}q_{2} \right\|_{L^{2}(\mathbb{S}^{1} \times \mathbb{S}^{1})}, \ q_{1}, q_{2} \in K.
\end{equation}
\end{lem}
Let $\{ \theta_{i}: i \in \mathbb{N} \}$ be dense in $\mathbb{S}^{1}$.
We denote $\vec{\mathcal{F}}_{N}:L^{2}(Q) \to L^{\infty}(\mathbb{S}^{1})^{N}$ by  
\begin{equation}
\vec{\mathcal{F}}_{N}(q):=\left(
\begin{array}{cc}
\mathcal{F}_{\theta_{1}}q \\
\vdots \\
\mathcal{F}_{\theta_{N}}q
\end{array}
\right).  
\end{equation}
Although Lemma \ref{Lemma 2.3} showed the Lipschitz stability with infinite-dimensional measurements, by applying Theorem 7 of \cite{alberti2020infinitedimensional}, we have the Lipschitz stability with finitely many measurements.

\begin{lem}\label{Lipchotz stability}
Let $W$ be a finite-dimensional subspace of $L^{2}(Q)$ and $K$ be a compact and convex subset of $W \cap L^{\infty}_{+}(Q)$.
Then, for large $N \in \mathbb{N}$, there exists a constant $C_{N}>0$ such that
\begin{equation}
\left\| q_{1} - q_{2} \right\|_{L^{2}(Q)} \leq C_{N} \left\| \vec{\mathcal{F}}_{N}q_{1} - \vec{\mathcal{F}}_{N}q_{2} \right\|_{L^{2}(\mathbb{S}^{1})^{N}}, \ q_{1}, q_{2} \in K. \label{stab}
\end{equation}
\end{lem}
Finally, we recall the Levenberg-Marquardt scheme for our problem as follows:
\begin{equation}
q_{i+1}=q_{i} + \left(\alpha_{i} I + \vec{\mathcal{F}}_{N}^{\prime}[q_{i}]^{*}\vec{\mathcal{F}}_{N}^{\prime}[q_{i}]  \right)^{-1}\vec{\mathcal{F}}_{N}^{\prime}[q_{i}]^{*}\left( \vec{u}^{\infty} - \vec{\mathcal{F}}_{N}(q_{i}) \right), \label{3.9}
\end{equation}
for $i \in \mathbb{N}_{0}$ (see, e.g., \cite{Hanke_1997, Kaltenbacher}). 
We remark that the regularization parameter $\alpha_i>0$ in (\ref{3.9}) must be chosen appropriately.
One of the choices of $\alpha_i$ is based on the Morozov discrepancy principle, that is, $\alpha_i>0$ is chosen such that it satisfies
\begin{equation}
\left\| \vec{u}^{\infty} - \vec{\mathcal{F}}_{N}(q_{i}) - \vec{\mathcal{F}}_{N}^{\prime}[q_{i}] (q_{i+1}(\alpha_i) - q_{i}) \right\| = \rho \left\| \vec{u}^{\infty} - \vec{\mathcal{F}}_{N}(q_{i}) \right\|, \label{morozov discrepancy principle}
\end{equation}
for some fixed $\rho \in (0, 1)$, where $q_{i+1}(\alpha_i)$ is defined by (\ref{3.9}). 
Applying Theorem 4.2 of \cite{Kaltenbacher} with Lemma \ref{Lipchotz stability}, we obtain the following convergence for (\ref{3.9}). 
\begin{lem} \label{convergence LM}
Let $W$ be a finite-dimensional subspace of $L^{2}(Q)$ and let $0<r<1$ be small such that $B_{r}(q_{0}) \Subset W \cap L^{\infty}_{+}(Q)$.
Let $N \in \mathbb{N}$ be large such that (\ref{stab}) holds, let $0<\rho<1$, and let $q^{\dag} \in B_{r}(q_{0})$ be a solution of $\vec{u}^{\infty} = \vec{\mathcal{F}}_{N}(q^{\dag})$.
Then, there exists a constant $C_{N}>0$ such that if $\left\| q_{0} - q^{\dag} \right\| < \rho / C_{N}$, then the Levenberg–Marquardt iteration $q_{i}$ in (\ref{3.9}) with the $\{ \alpha_{i}\}_{i \in \mathbb{N}_{0}}$ defined in (\ref{morozov discrepancy principle}) converges to the solution $q^{\dag}$ as $i \to \infty$.
\end{lem}

\subsection{Kalman filter}\label{Kalman filter}
In this section, we review the Kalman filter in a general functional analytic setting.
Let $X$ and $Y$ be Hilbert spaces over complex variables $\mathbb{C}$, $f_n \in Y$ ($n=1,...,N$) be a measurement, and $A_{n}:X \to Y$ ($n=1,...,N$) be a linear operator from $X$ to $Y$.
We consider the problem of determining $\varphi \in X$ such that 
\begin{equation}
A_n \varphi = f_n, \label{3.1}
\end{equation}
for all $n=1,...,N$. 
Now, we assume that we have the initial guess $\varphi_0 \in X$, which is the starting point of the algorithm, and that it was appropriately determined by a priori information of the true solution $\varphi^{true}$. 
Then, we consider the minimization problem of the following functional.
\begin{eqnarray}
J_{Full, N}(\varphi)&:=&\alpha \left\| \varphi - \varphi_0 \right\|^{2}_{X} + \left\| \vec{f} - \vec{A}\varphi \right\|^{2}_{Y^{N},  R^{-1}}
\nonumber\\
&=&\alpha \left\| \varphi - \varphi_0 \right\|^{2}_{X} + \sum_{n=1}^{N}\left\| f_n - A_n\varphi \right\|^{2}_{Y, R^{-1}}, \label{3.2}
\end{eqnarray}
where $\vec{f}:=\left(
    \begin{array}{cc}
      f_1 \\
      \vdots \\
      f_N
    \end{array}
  \right)$ and $\vec{A}:=\left(
    \begin{array}{cc}
      A_1 \\
      \vdots \\
      A_N
    \end{array}
  \right)$. 
The norm $\left\| \cdot \right\|^{2}_{Y, R^{-1}}:=\langle \cdot, R^{-1} \cdot \rangle_{Y}$ is a weighted norm with a positive definite symmetric invertible operator $R: Y \to Y$. 
The minimizer of (\ref{3.2}) is given by
\begin{equation}
\varphi^{FT}_{N} := \varphi_0 + (\alpha I + \vec{A}^{*}\vec{A})^{-1}\vec{A}^{*}\left(\vec{f} - \vec{A}\varphi_0 \right). \label{3.5}
\end{equation}
We call this the {\it Full data Tikhonov}. 
Here, $\vec{A}^{*}$ is the adjoint operator with respect to $\langle \cdot, \cdot \rangle_{X}$ and $\langle \cdot, \cdot \rangle_{Y^{N}, R^{-1}}$.
We calculate
\begin{eqnarray}
\langle \vec{f}, \vec{A} \varphi \rangle_{Y^N, R^{-1}} &=& \sum_{n=1}^{N} \langle f_n, R^{-1}A_n \varphi \rangle_{Y}
\nonumber\\
&=&\sum_{n=1}^{N} \langle A^{H}_n R^{-1} f_n, \varphi \rangle_{X} = \langle  \vec{A}^{H} R^{-1} \vec{f}, \varphi \rangle_{X}, \label{3.6}
\end{eqnarray}
which implies that 
\begin{equation}
\vec{A}^{*}=\vec{A}^{H} R^{-1}, \label{3.7}
\end{equation}
where $A_{n}^{H}$ and $\vec{A}^{H}$ are the adjoint operators with respect to the usual scalar products $\langle \cdot, \cdot \rangle_{X}$, $\langle \cdot, \cdot \rangle_{Y}$ and $\langle \cdot, \cdot \rangle_{X}$, $\langle \cdot, \cdot \rangle_{Y^{N}}$, respectively.
Then, the Full data Tikhonov solution in (\ref{3.5}) is of the form
\begin{equation}
\varphi^{FT}_{N} = \varphi_0 + \left( \alpha I + \vec{A}^{H} R^{-1}\vec{A} \right)^{-1}\vec{A}^{H} R^{-1}\left(\vec{f} - \vec{A}\varphi_0 \right). \label{3.8}
\end{equation}
\par
However, algorithm (\ref{3.8}) of the Full data Tikhonov is computationally expensive because we must construct a larger vector $\vec{f}$ and a large operator $\vec{A}$ when the number of measurements $N$ increases.
Accordingly, we consider the alternative approach based on the Kalman filter (see, e.g., \cite{grewal2010applications, Grewal, Jazwinski}), which is the algorithm give by the following algorithm:
\begin{equation}
\varphi^{KF}_{n}:= \varphi^{KF}_{n-1} + K_{n}\left( f_{n}-A_{n} \varphi^{KF}_{n-1} \right), \label{4.21}
\end{equation}
\begin{equation}
K_{n}:= B_{n-1} A^{H}_{n}\left(R + A_{n} B_{n-1} A^{H}_{n} \right)^{-1}, \label{4.22}
\end{equation}
\begin{equation}
B_{n}:= \left(I - K_{n} A_{n} \right)B_{n-1}, \label{4.23}
\end{equation}
for $n=1,...,N$, where $\varphi^{KF}_{0}:=\varphi_0$ and $B_{0}:=\frac{1}{\alpha}I$.
Here, $\varphi^{KF}_{n}$ is the unique minimizer of the following functional (see Section 7 of \cite{Freitag} and Section 5 of \cite{NakamuraPotthast}):
\begin{equation}
J_{KF,n}(\varphi):= \left\| \varphi - \varphi^{KF}_{n-1} \right\|^{2}_{X, B_{n-1}^{-1}} + \left\| f_n - A_n \varphi \right\|^{2}_{Y, R^{-1}}.
\end{equation}
\par
We observe that the Kalman filter algorithm updates state $\varphi$ every $n$ with measurement $f_{n}$ and one operator $A_n$, and it does not require large vectors or operators.
Instead, it updates both the state $\varphi$ in (\ref{4.21}) and weight $B$ of the norm in (\ref{4.23}), which plays the role of retaining the information from past updates. 
By the same argument in Theorem 5.4.7 of \cite{NakamuraPotthast}, we can prove the equivalence of the Full data Tikhonov and Kalman filter when all observation operators $A_n$ are linear.
\begin{lem} \label{equivalence}
For measurements $f_1,...,f_N$, linear operators $A_1,...,A_N$, and the initial guess $\varphi_0 \in X$, the final state of the Kalman filter given by (\ref{4.21})--(\ref{4.23}) is equivalent to the state of the Full data Tikhonov given by (\ref{3.8}), that is,
\begin{equation}
\varphi^{KF}_{N}=\varphi^{FT}_{N}. \label{4.24}
\end{equation}
\end{lem}
\section{Kalman filter Levenberg–Marquardt}\label{Kalman filter Levenberg–Marquardt}
In this section, we propose a reconstruction algorithm based on the Kalman filter that is equivalent to the Levenberg–Marquardt algorithm.
We solve the following problem with respect to $q$: 
\begin{equation}
\mathcal{F}_{\theta_{n}}q=u^{\infty}_{\theta_n}, \ n=1,...,N. \label{IMSP}
\end{equation}
It is convenient to employ the vector notation as follows:
\begin{equation}
\vec{\mathcal{F}}q =\vec{u}^{\infty}, \label{IMSP-System}
\end{equation}
where 
$
\vec{\mathcal{F}}q:=\left(
\begin{array}{cc}
\mathcal{F}_{\theta_{1}}q \\
\vdots \\
\mathcal{F}_{\theta_{N}}q
\end{array}
\right)$ and 
$\vec{u}^{\infty}:=\left(
\begin{array}{cc}
u^{\infty}_{\theta_1} \\
\vdots \\
u^{\infty}_{\theta_N}
\end{array}
\right)
$.
\par
First, we review the derivation of the Levenberg-Marquardt scheme (\ref{3.9}) as follows.
We assume that we have an initial guess $q_{0}$ and consider the Taylor expansion at $q=q_0$.
\begin{equation}
\vec{\mathcal{F}}q = \vec{\mathcal{F}}q_0 + \vec{\mathcal{F}}^{\prime}[q_0](q - q_0) + r(q - q_0). \label{linearized problem}
\end{equation}
We forget the high-order term $r(q - q_0)$ and solve the linearized problem for (\ref{IMSP-System}):
\begin{equation}
\vec{\mathcal{F}}^{\prime}[q_0]q = \vec{u}^{\infty} - \vec{\mathcal{F}}q_0 +\vec{\mathcal{F}}^{\prime}[q_0]q_0,
\end{equation}
where
$
\vec{\mathcal{F}}^{\prime}[q_0]q:=\left(
\begin{array}{cc}
\mathcal{F}^{\prime}_{\theta_{1}}[q_0]q \\
\vdots \\
\mathcal{F}^{\prime}_{\theta_{N}}[q_0]q
\end{array}
\right) 
$.
Then, the Tikhonov regularization solution is given by
\begin{equation}
q_{1} := q_{0} + \left(\alpha_{0} I + \vec{\mathcal{F}}^{\prime}[q_{0}]^{*}\vec{\mathcal{F}}^{\prime}[q_{0}] \right)^{-1}\vec{\mathcal{F}}^{\prime}[q_{0}]^{*}\left(\vec{u}^{\infty} - \vec{\mathcal{F}}q_0 \right),
\end{equation}
where $\alpha_0 >0$ is a regularization parameter.
\par
Next, we solve the linearized problem for (\ref{IMSP-System}) with initial guess $q_{1}$:
\begin{equation}
\vec{\mathcal{F}}^{\prime}[q_1]q = \vec{u}^{\infty} - \vec{\mathcal{F}}q_1 +\vec{\mathcal{F}}^{\prime}[q_1]q_1.
\end{equation}
Then, the Tikhonov regularization solution is given by 
\begin{equation}
q_{2} := q_{1} + \left(\alpha_{1} I + \vec{\mathcal{F}}^{\prime}[q_{1}]^{*}\vec{\mathcal{F}}^{\prime}[q_{1}] \right)^{-1}\vec{\mathcal{F}}^{\prime}[q_{1}]^{*}\left(\vec{u}^{\infty} - \vec{\mathcal{F}}q_1 \right), \label{end LM}
\end{equation}
where $\alpha_1 >0$ is a regularization parameter.
Repeating the above arguments (\ref{linearized problem})--(\ref{end LM}), we have the iteration scheme for $i\in \mathbb{N}_0$:
\begin{equation}
q^{FLM}_{i+1} := q^{FLM}_{i} + \left(\alpha_{i} I + \vec{\mathcal{F}}^{\prime}[q^{FLM}_{i}]^{*}\vec{\mathcal{F}}^{\prime}[q^{FLM}_{i}] \right)^{-1} \vec{\mathcal{F}}^{\prime}[q^{FLM}_{i}]^{*}\left(\vec{u}^{\infty} - \vec{\mathcal{F}}q^{FLM}_i \right), \label{FTN}
\end{equation}
where $\{\alpha_{i} \}_{i \in \mathbb{N}_0}$ is a sequence of regularization parameters.
We call this the {\it Full data Levenberg–Marquardt} (FLM). 
Here, $\vec{\mathcal{F}}^{\prime}[q^{FLM}_{i}]^{*}$ is an adjoint operator of $\vec{\mathcal{F}}^{\prime}[q^{FLM}_{i}]$ with respect to the usual scalar product $\langle \cdot, \cdot \rangle_{L^{2}(Q)}$ and weighted scalar product $\langle \cdot, \cdot \rangle_{L^{2}(\mathbb{S}^{1})^{N}, R^{-1}}$, where $R: L^{2}(\mathbb{S}^{1}) \to L^{2}(\mathbb{S}^{1})$ is the positive definite symmetric invertible linear operator. 
By the same calculation as in (\ref{3.6}), we have
\begin{equation}
\vec{\mathcal{F}}^{\prime}[q^{FLM}_{i}]^{*}=\vec{\mathcal{F}}^{\prime}[q^{FLM}_{i}]^{H} R^{-1},
\end{equation}
where $\vec{\mathcal{F}}^{\prime}[q^{LM}_{i}]^{H}$ is an adjoint operator of $\vec{\mathcal{F}}^{\prime}[q^{LM}_{i}]$ with respect to usual scalar products $\langle \cdot, \cdot \rangle_{L^{2}(Q)}$ and $\langle \cdot, \cdot \rangle_{L^{2}(\mathbb{S}^{1})^{N}}$. 
Then, (\ref{FTN}) can be of the form
\begin{equation}
q^{FLM}_{i+1} = q^{FLM}_{i} + \left(\alpha_{i} I + \vec{\mathcal{F}}^{\prime}[q^{LM}_{i}]^{H} R^{-1}\vec{\mathcal{F}}^{\prime}[q^{LM}_{i}]\right)^{-1}\vec{\mathcal{F}}^{\prime}[q^{LM}_{i}]^{H} R^{-1}\left(\vec{u}^{\infty} - \vec{\mathcal{F}}q^{FLM}_{i} \right). \label{FTN-H}
\end{equation}
As stated in Section \ref{Kalman filter}, algorithm (\ref{FTN-H}) is computationally expensive when the number of measurements $N$ increases.
Accordingly, we consider the alternative approach based on the Kalman filter.
\par
We denote
\begin{equation}
q_{0,0} := q_{0} \ \ \mbox{and} \ \ B_{0,0} := \frac{1}{\alpha_{0}}I. \label{start KFL}
\end{equation}
We solve the linearized problem for (\ref{IMSP}) with initial guess $q_{0,0}$:
\begin{equation}
\mathcal{F}^{\prime}_{\theta_{n}}[q_{0,0}]q =u^{\infty}_{\theta_{n}} - \mathcal{F}_{\theta_{n}}q_{0,0} + \mathcal{F}^{\prime}_{\theta_{n}}[q_{0,0}]q_{0,0}, \label{linearizedeq1}
\end{equation}
for $n=1,...,N$.
Applying the Kalman filter update (\ref{4.21})--(\ref{4.23}) as 
\begin{equation}
f_{n}=u^{\infty}_{\theta_{n}} - \mathcal{F}_{\theta_{n}}q_{0,0} + \mathcal{F}^{\prime}_{\theta_{n}}[q_{0,0}]q_{0,0}, \ \  A_{n} = \mathcal{F}^{\prime}_{\theta_{n}}[q_{0,0}], \ \ \varphi_{0} = q_{0,0}, \ \ \mbox{and} \ \ B_{0}=B_{0,0},
\end{equation}
we obtain the algorithm for $n=1,...,N$,
\begin{equation}
q_{0, n}:= q_{0, n-1} + K_{0, n}\left( u^{\infty}_{\theta_{n}} - \mathcal{F}_{\theta_{n}}q_{0,0} + \mathcal{F}^{\prime}_{\theta_{n}}[q_{0,0}]q_{0,0} - \mathcal{F}_{\theta_{n}}[q_{0,0}] q_{0, n-1} \right),
\end{equation}
\begin{equation}
K_{0, n}:= B_{0, n-1}\mathcal{F}^{\prime}_{\theta_{n}}[q_{0,0}]^{H}\left(R + \mathcal{F}^{\prime}_{\theta_{n}}[q_{0,0}] B_{0, n-1} \mathcal{F}^{\prime}_{\theta_{n}}[q_{0,0}]^{H} \right)^{-1}, 
\end{equation}
\begin{equation}
B_{0, n}:= \left(I - K_{0, n} \mathcal{F}^{\prime}_{\theta_{n}}[q_{0,0}] \right)B_{0, n-1}.
\end{equation}
\par
Next, we denote
\begin{equation}
q_{1,0} := q_{0, N} \ \ \mbox{and} \ \ B_{1,0} := \frac{1}{\alpha_{1}}I,    
\end{equation}
We solve the linearized problem for (\ref{IMSP}) with initial guess $q_{1,0}$:
\begin{equation}
\mathcal{F}^{\prime}_{\theta_{n}}[q_{1,0}]q = u^{\infty}_{\theta_{n}} - \mathcal{F}_{\theta_{n}}q_{1,0} + \mathcal{F}^{\prime}_{\theta_{n}}[q_{1,0}]q_{1,0},
\end{equation}
for $n=1,...,N$. Applying the Kalman filter update (\ref{4.21})--(\ref{4.23}) as 
\begin{equation}
f_{n}=u^{\infty}_{\theta_{n}} - \mathcal{F}_{\theta_{n}}q_{1,0} + \mathcal{F}^{\prime}_{\theta_{n}}[q_{1,0}]q_{1,0}, \ \  A_{n} = \mathcal{F}^{\prime}_{\theta_{n}}[q_{1,0}], \ \ \varphi_{0} = q_{1,0}, \ \ \mbox{and} \ \ B_{0}=B_{1,0},
\end{equation}
we obtain the algorithm for $n=1,...,N$:
\begin{equation}
q_{1, n}:= q_{1, n-1} + K_{1, n}\left( u^{\infty}_{\theta_{n}} - \mathcal{F}_{\theta_{n}}q_{1,0} + \mathcal{F}^{\prime}_{\theta_{n}}[q_{1,0}]q_{1,0} - \mathcal{F}_{\theta_{n}}[q_{1,0}] q_{1, n-1} \right),
\end{equation}
\begin{equation}
K_{1, n}:= B_{1, n-1}\mathcal{F}^{\prime}_{\theta_{n}}[q_{1,0}]^{H}\left(R + \mathcal{F}^{\prime}_{\theta_{n}}[q_{1,0}] B_{1, n-1} \mathcal{F}^{\prime}_{\theta_{n}}[q_{1,0}]^{H} \right)^{-1},  
\end{equation}
\begin{equation}
B_{1, n}:= \left(I - K_{1, n} \mathcal{F}^{\prime}_{\theta_{n}}[q_{1,0}] \right)B_{1, n-1}.\label{end KFL}
\end{equation}
Repeating the above arguments (\ref{start KFL})--(\ref{end KFL}), we obtain the following algorithm for $n=1,...,N$ and $i \in \mathbb{N}_0$:
\begin{equation}
q^{KFL}_{i, n}:= q^{KFL}_{i, n-1} + K_{i, n}\left( u^{\infty}_{\theta_{n}} - \mathcal{F}_{\theta_{n}}q^{KFL}_{i,0} + \mathcal{F}^{\prime}_{\theta_{n}}[q^{KFL}_{i,0}]q^{KFL}_{i,0} - \mathcal{F}_{\theta_{n}}[q^{KFL}_{i,0}] q^{KFL}_{i, n-1} \right),  \label{KFL-1}
\end{equation}
\begin{equation}
K_{i, n}:= B_{i, n-1}\mathcal{F}^{\prime}_{\theta_{n}}[q^{KFL}_{i,0}]^{H}\left(R + \mathcal{F}^{\prime}_{\theta_{n}}[q^{KFL}_{i,0}] B_{i, n-1} \mathcal{F}^{\prime}_{\theta_{n}}[q^{KFL}_{i,0}]^{H} \right)^{-1},  \label{KFL-2}
\end{equation}
\begin{equation}
B_{i, n}:= \left(I - K_{i, n} \mathcal{F}^{\prime}_{\theta_{n}}[q^{KFL}_{i, 0}] \right)B_{i, n-1}, \label{KFL-3}
\end{equation}
where 
\begin{equation}
q^{KFL}_{i, 0}:=q^{KFL}_{i-1,N},
\end{equation}
\begin{equation}
B_{i,0}:=\frac{1}{\alpha_{i}}I. \label{KFL-B}
\end{equation}
We call this the {\it Kalman filter Levenberg–Marquardt} (KFL).
We remark that the algorithm has indexes $i$ and $n$, where $i$ is associated with the iteration step and $n$ with the measurement step, respectively.
\par
Finally, we show that the KFL is equivalent to the FLM.
\begin{thm}\label{equivalence for KFN and FTN}
For the initial guess $q_{0} \in L^{2}(Q)$ and sequence $\{\alpha_{i}\}_{i \in \mathbb{N}_{0}}$ of the regularization parameters, the Kalman filter Levenberg–Marquardt (\ref{KFL-1})--(\ref{KFL-B}) is equivalent to the Full data Levenberg–Marquardt given by (\ref{FTN-H}), that is, for all $i \in \mathbb{N}_0$, we have
\begin{equation}
q^{KFL}_{i, N}=q^{FLM}_{i+1}.  \label{eqiv-KFL}
\end{equation}
\end{thm}
\begin{proof}
We prove (\ref{eqiv-KFL}) by induction.
Applying Lemma \ref{equivalence} to the linearized problem (\ref{linearizedeq1}) with the regularization parameter $\alpha_0$, we have 
\begin{equation}
q^{KFL}_{0, N}=q^{FLM}_{1},
\end{equation}
which is the base case $i=0$.
\par
Assume that (\ref{eqiv-KFL}) $i-1$ holds, that is, we have 
\begin{equation}
q^{KFL}_{i-1, N}=q^{FLM}_{i}=:q_{i}.
\end{equation}
Again, applying Lemma \ref{equivalence} to the linearized problem ($n=1,...,N$)
\begin{equation}
\mathcal{F}^{\prime}_{\theta_{n}}[q_{i}]q =u^{\infty}_{\theta_{n}} -  \mathcal{F}_{\theta_{n}}q_{i} + \mathcal{F}^{\prime}_{\theta_{n}}[q_{i}]q_{i}, \label{linearizedeq2}
\end{equation}
with the regularization parameter $\alpha_{i-1}$, we have
\begin{equation}
q^{KFL}_{i, N}=q^{FLM}_{i+1}.
\end{equation}
Theorem \ref{equivalence for KFN and FTN} has been shown.
\end{proof}
\begin{rem}
We note that by the equivalence with FLM (Theorem \ref{equivalence for KFN and FTN}) and the convergence of Levenberg-Marquardt (Lemma \ref{convergence LM}), the KFL algorithm (\ref{KFL-1})--(\ref{KFL-3}) converges under some assumption.
\end{rem}
\par
Figure \ref{KFL-fig} illustrates the FLM and KFL.
While the FLM only moves horizontally, the KFL first moves vertically.
Once it was moved up to $n=N$, it moves horizontally, and then the linearization is complete.
\begin{figure}[h]
\centering
\includegraphics[keepaspectratio, scale=0.7]{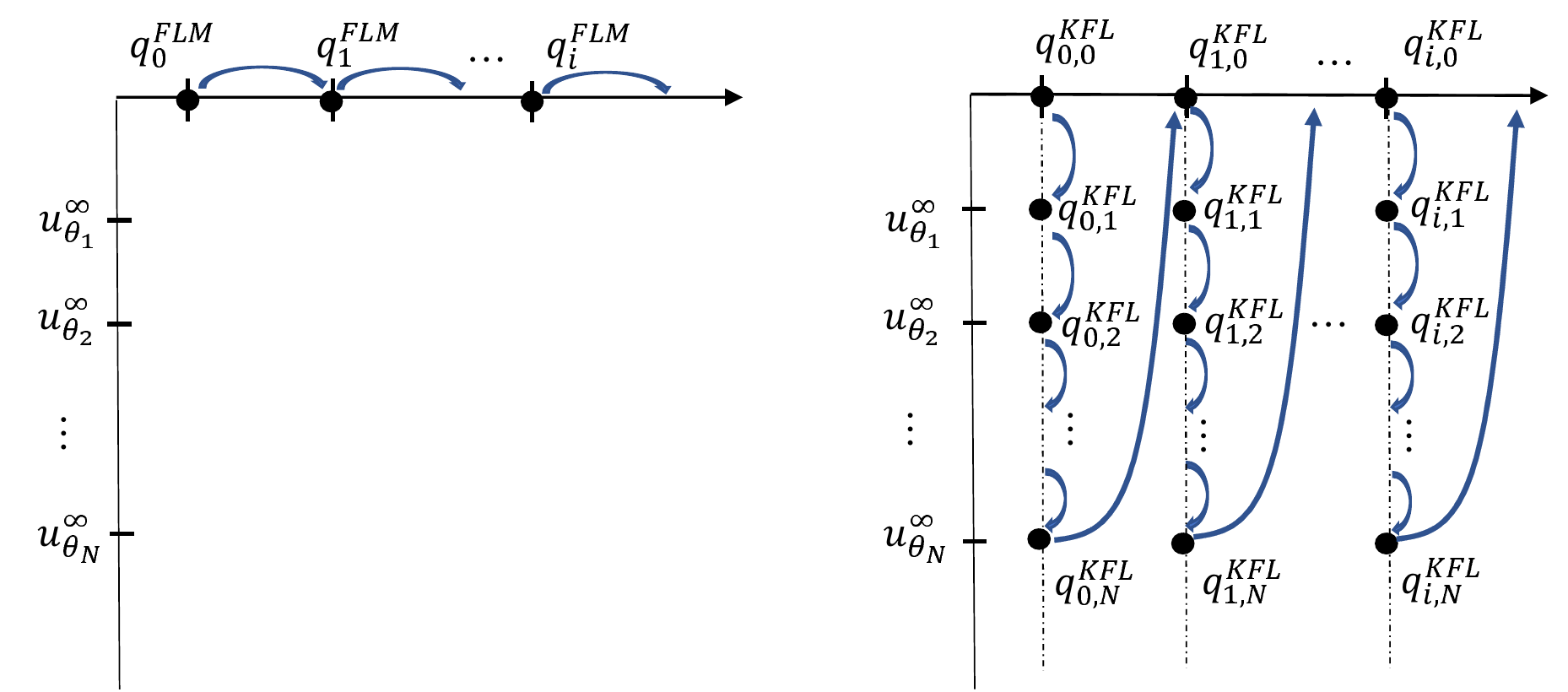}
\caption{Illustration of FLM (left) and KFL (right)}
\label{KFL-fig}
\end{figure}
\section{Iterative Extended Kalman filter}\label{Iterative Extended Kalman filter}
In this section,
we propose the algorithm inspired by the Extended Kalman filter (see, e.g., \cite{grewal2010applications, Grewal, Jazwinski}). 
We denote
\begin{equation}
q_{0,0} := q_{0} \ \ \mbox{and} \ \ B_{0,0} := \frac{1}{\alpha_{0}}I. \label{start EKF}
\end{equation}
We solve the linearized problem for the equation
\begin{equation}
u^{\infty}_{\theta_{1}} =\mathcal{F}_{\theta_{1}}q,
\end{equation}
with respect to $n=1$ and initial guess $q_{0,0}$, that is,
\begin{equation}
\mathcal{F}^{\prime}_{\theta_{1}}[q_{0,0}]q = u^{\infty}_{\theta_{1}} - \mathcal{F}_{\theta_{1}}q_{0,0} + \mathcal{F}^{\prime}_{\theta_{1}}[q_{0,0}]q_{0,0},
\end{equation}
which is equivalent to solving the minimization problem of the following functional:
\begin{equation}
J_{0,0}(q):= \left\| q - q_{0,0} \right\|^{2}_{X, B_{0,0}^{-1}} + \left\| u^{\infty}_{\theta_{1}} - \mathcal{F}_{\theta_{1}}q_{0,0} + \mathcal{F}^{\prime}_{\theta_{1}}[q_{0,0}]q_{0,0} - \mathcal{F}^{\prime}_{\theta_{1}}[q_{0,0}] q \right\|^{2}_{Y, R^{-1}}.
\end{equation}
By the same argument as in Section 7 of \cite{Freitag}, the Tikhonov regularization solution has the following form:
\begin{equation}
q_{0,1}:= q_{0,0} + K_{0,1}\left( u^{\infty}_{\theta_{1}} - \mathcal{F}_{\theta_{1}}q_{0,0} \right),
\end{equation}
\begin{equation}
K_{0,1}:= B_{0,0} \mathcal{F}^{\prime}_{\theta_{1}}[q_{0,0}]^{H}\left(R + \mathcal{F}^{\prime}_{\theta_{1}}[q_{0,0}] B_{0,0} \mathcal{F}^{\prime}_{\theta_{1}}[q_{0,0}]^{H} \right)^{-1}, 
\end{equation}
\begin{equation}
B_{0,1}:= \left(I - K_{0,1} \mathcal{F}^{\prime}_{\theta_{1}}[q_{0,0}] \right)B_{0,0}.
\end{equation}
\par 
Next, we solve the linearized problem for the equation
\begin{equation}
u^{\infty}_{\theta_{2}} =\mathcal{F}_{\theta_{2}}q,
\end{equation}
with respect to $n=2$ and initial guess $q_{0,1}$, that is,
\begin{equation}
\mathcal{F}^{\prime}_{\theta_{2}}[q_{0,1}]q = u^{\infty}_{\theta_{2}} - \mathcal{F}_{\theta_{2}}q_{0,1} + \mathcal{F}^{\prime}_{\theta_{2}}[q_{0,1}]q_{0,1},
\end{equation}
which is equivalent to solving the minimization problem of the following functional:
\begin{equation}
J_{0,1}(q):= \left\| q - q_{0,1} \right\|^{2}_{X, B_{0,1}^{-1}} + \left\| u^{\infty}_{\theta_{2}} - \mathcal{F}_{\theta_{2}}q_{0,1} + \mathcal{F}^{\prime}_{\theta_{2}}[q_{0,1}]q_{0,1} - \mathcal{F}^{\prime}_{\theta_{2}}[q_{0,1}] q \right\|^{2}_{Y, R^{-1}}.
\end{equation}
The Tikhonov regularization solution has the following form:
\begin{equation}
q_{0,2}:= q_{0,1} + K_{0,2}\left( u^{\infty}_{\theta_{2}} - \mathcal{F}_{\theta_{2}}q_{0,1} \right),
\end{equation}
\begin{equation}
K_{0,2}:= B_{0,1} \mathcal{F}^{\prime}_{\theta_{2}}[q_{0,1}]^{H}\left(R + \mathcal{F}^{\prime}_{\theta_{2}}[q_{0,1}] B_{0,1} \mathcal{F}^{\prime}_{\theta_{2}}[q_{0,1}]^{H} \right)^{-1},
\end{equation}
\begin{equation}
B_{0,2}:= \left(I - K_{0,2} \mathcal{F}^{\prime}_{\theta_{2}}[q_{0,1}] \right)B_{0,1}. \label{end EKF}
\end{equation}
Repeating the above arguments (\ref{start EKF})--(\ref{end EKF}), we obtain the algorithm for $n=1,...,N$:
\begin{equation}
q_{0,n}:= q_{0,n-1} + K_{0,n-1}\left( u^{\infty}_{\theta_{n}} - \mathcal{F}_{\theta_{n}}q_{n-1} \right), 
\end{equation}
\begin{equation}
K_{0,n}:= B_{0,n-1} \mathcal{F}^{\prime}_{\theta_{n}}[q_{0,n-1}]^{H}\left(R + \mathcal{F}^{\prime}_{\theta_{n}}[q_{0,n-1}] B_{0,n-1} \mathcal{F}^{\prime}_{\theta_{n}}[q_{0,n-1}]^{H} \right)^{-1}, 
\end{equation}
\begin{equation}
B_{0,n}:= \left(I - K_{0,n} \mathcal{F}^{\prime}_{\theta_{n}}[q_{0,n-1}] \right)B_{0,n-1}.
\end{equation}
\par
Next, we denote
\begin{equation}
q_{1,0}:= q_{0,N}, \ \ \mbox{and} \ \ B_{1,0} := \frac{1}{\alpha_{1}}I.
\end{equation}
We solve the linearized problem for $u^{\infty}_{\theta_{1}} =\mathcal{F}_{\theta_{1}}q$ with respect to $n=1$ and initial guess $q_{1,0}$, that is,
\begin{equation}
\mathcal{F}^{\prime}_{\theta_{1}}[q_{1,0}]q=u^{\infty}_{\theta_{1}} - \mathcal{F}_{\theta_{1}}q_{1,0} + \mathcal{F}^{\prime}_{\theta_{1}}[q_{1,0}]q_{1,0}.
\end{equation}
The Tikhonov regularization solution has the following form:
\begin{equation}
q_{1,1}:= q_{1,0} + K_{1}\left( u^{\infty}_{\theta_{1}} - \mathcal{F}_{\theta_{1}}q_{1,0} \right), 
\end{equation}
\begin{equation}
K_{1,1}:= B_{1,0} \mathcal{F}^{\prime}_{\theta_{1}}[q_{1,0}]^{H}\left(R + \mathcal{F}^{\prime}_{\theta_{1}}[q_{1,0}] B_{1,0} \mathcal{F}^{\prime}_{\theta_{1}}[q_{1,0}]^{H} \right)^{-1}, 
\end{equation}
\begin{equation}
B_{1,1}:= \left(I - K_{1,1} \mathcal{F}^{\prime}_{\theta_{1}}[q_{1,0}] \right)B_{1,0}.
\end{equation}
By solving the linearized problem for $u^{\infty}_{\theta_{n}} =\mathcal{F}_{\theta_{n}}q$ up to $n=N$, we obtain $q_{1,N}$, $B_{1,N}$, and denote $q_{2,0}:=q_{1,N}$, $B_{2,0}:= \frac{1}{\alpha_{2}}I,$.
Repeating the above arguments, we finally obtain the following algorithm for $n=1,...,N$ and $i \in \mathbb{N}_0$:
\begin{equation}
q^{EKF}_{i, n}:= q^{EKF}_{i, n-1} + K_{i, n-1}\left( u^{\infty}_{\theta_{n}} - \mathcal{F}_{\theta_{n}}q^{EKF}_{i, n-1} \right), \label{EKF-1}
\end{equation}
\begin{equation}
K_{i, n}:= B_{i,n-1} \mathcal{F}^{\prime}_{\theta_{n}}[q^{EKF}_{i,n-1}]^{H}\left(R + \mathcal{F}^{\prime}_{\theta_{n}}[q^{EKF}_{i,n-1}] B_{i,n-1} \mathcal{F}^{\prime}_{\theta_{n}}[q^{EKF}_{i,n-1}]^{H} \right)^{-1}, \label{EKF-2}
\end{equation}
\begin{equation}
B_{i,n}:= \left(I - K_{i,n} \mathcal{F}^{\prime}_{\theta_{n}}[q^{EKF}_{i,n-1}] \right)B_{i,n-1}, \label{EKF-3}
\end{equation}
where
\begin{equation}
q^{EKF}_{i, 0}:=q^{EKF}_{i-1,N},
\end{equation}
\begin{equation}
B_{i,0}:=\frac{1}{\alpha_{i}}I. \label{EKF-B}
\end{equation}
We call this the {\it iterative Extended Kalman filter} (EKF).
As remarked in Section \ref{Kalman filter Levenberg–Marquardt}, the algorithm has indexes $i$ and $n$, where $i$ is associated with the iteration step and $n$ with the measurement step, respectively.
Figure \ref{EKF-fig} illustrates the EKF. 
EKF always moves diagonally because linearization is performed in every measurement step.
\begin{figure}[h]
\centering
\includegraphics[keepaspectratio, scale=0.6]{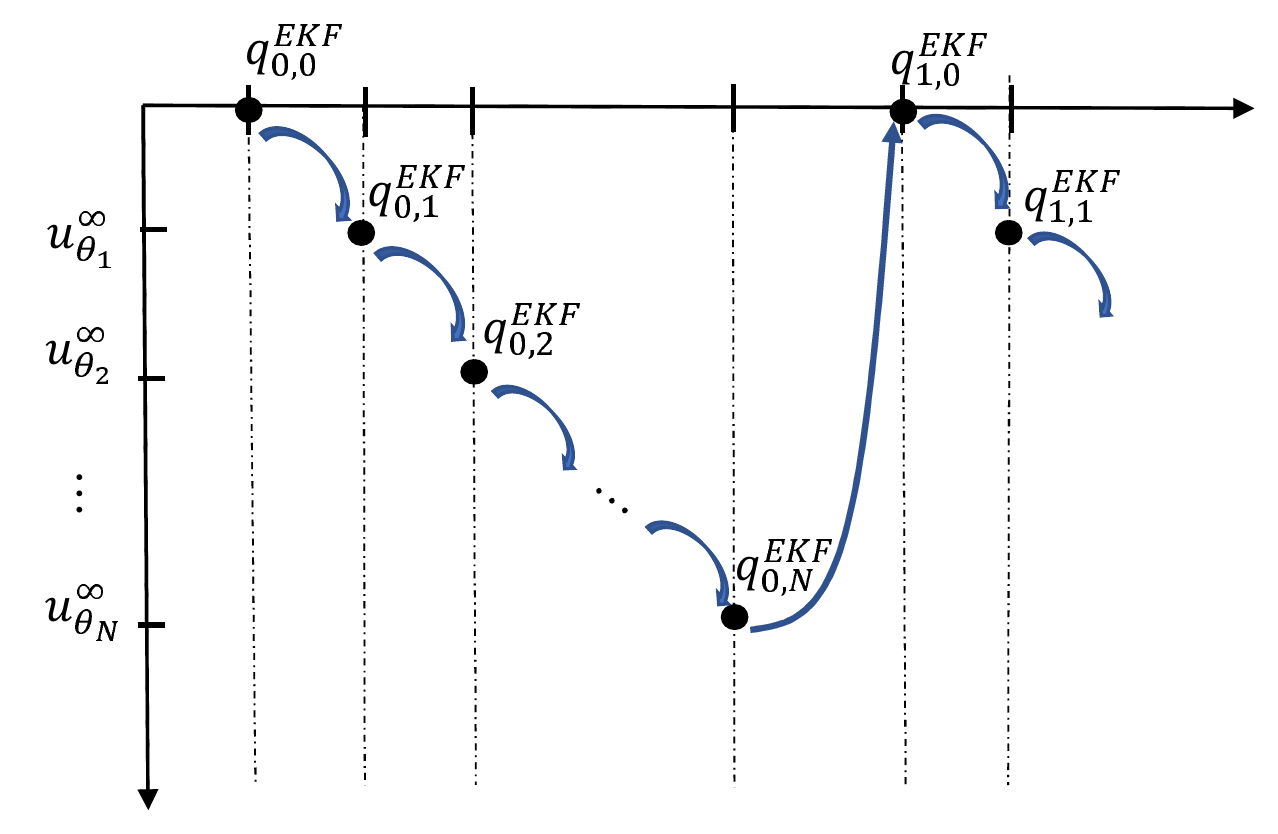}
\caption{Illustration of EKF}
\label{EKF-fig}
\end{figure}
\begin{rem}\label{EKF and KFL}
We compare the KFL with the EKF.
KFL is based on the linearization at the initial state for each iteration step, whereas EKF is based on the linearization at the current state for every iteration step, implying that the update of the KFL is slower than that of the EKF.
\end{rem}
\begin{rem}\label{update B}
In both the KFL and EKF algorithms, instead of initializations (\ref{KFL-B}) and (\ref{EKF-B}), we can update the weight of the norm for each iteration step, that is,
\begin{equation}
B_{i,0}:=B_{i-1,N}, \label{update weight B}
\end{equation}
which plays a role in retaining the information of past updates as iteration step $i$ proceeds. 
\end{rem}
\section{Numerical examples}\label{Numerical examples}
In this section, we provide numerical examples to demonstrate the algorithms.
The inverse scattering problem concerns solving the nonlinear integral equation for $n=1,...,N$ 
\begin{equation}
\mathcal{F}_{\theta_n}q=u^{\infty}(\cdot, \theta_n), \label{Nonlinear integral equation}
\end{equation}
where the operator $\mathcal{F}_{\theta_n}:L^{2}(Q) \to L^{\infty}(\mathbb{S}^{1})$ is defined by 
\begin{equation}
\mathcal{F}_{\theta_n}q(\hat{x})
=
\frac{k^2\mathrm{e}^{\frac{i\pi}{4}} }{\sqrt{8\pi k}}
\int_{Q}\mathrm{e}^{-ik\hat{x} \cdot y}u_{q}(y, \theta_{n}) q(y)dy, 
\end{equation}
where the incident direction is denoted by $\theta_n:=\left(\mathrm{cos}(2\pi n/N), \mathrm{sin}(2\pi n/N) \right)$. 
Here, $u_{q}(\cdot, \theta_{n})$ is the solution of the Lippmann-Schwinger integral equation (\ref{1.4}), which is numerically calculated based on Vainikko's method \cite{saranen2001periodic, vainikko2000fast}.
This is a fast solution method for the Lippmann–Schwinger equation based on periodization, fast Fourier
transform techniques, and multi-grid methods.
We assume that the support of function $q$ is included in $[-S, S]^2$ with some $S>0$, and function $q$ is discretized by a piecewise constant on $[-S, S]^{2}$ decomposed by squares with length $\frac{S}{M}$, that is,
\begin{equation}
q \approx \left( q(y_{m_{1},m_{2}})  \right)_{-M \leq m_{1},m_{2} \leq M-1} \in \mathbb{C}^{(2M)^2},
\end{equation} 
where $y_{m_{1},m_{2}}:=\left( \frac{(2m_{1}+1)S}{2M}, \frac{(2m_{2}+1)S}{2M} \right)$, and $M \in \mathbb{N}$ is a number of the division of $[0,S]$.
Furthermore, the function $u^{\infty}(\cdot, \theta_n)$ is discretized by
\begin{equation}
u_{n}^{\infty}(\cdot, \theta_n) \approx \left( u^{\infty}(\hat{x}_j, \theta_{n}) \right)_{j=1,...,J}  \in \mathbb{C}^{J}, 
\end{equation}
where $\hat{x}_j:=\left(\mathrm{cos}(2\pi j/J),  \mathrm{sin}(2\pi j/J) \right)$, and $J \in \mathbb{N}$ is a number of the division of $[0,2\pi]$.
\par
We always fix the following parameters as $J=60$, $M=6$, $S=3$, $N=60$, and $k=7$. 
We consider true functions as the characteristic function 
\begin{equation}
q^{true}_{j}(x):=\left\{ \begin{array}{ll}
1.0 & \quad \mbox{for $x \in B_j$}  \\
0 & \quad \mbox{for $x \notin B_j$}
\end{array} \right.,  \label{6.6}
\end{equation} 
where the support $B_j$ of the true function is considered in the following two types:
\begin{equation} 
B_1:=\left\{(x_1, x_2) : x^{2}_{1}+x^{2}_{2} < 1.0   \right\},  \label{6.7}
\end{equation}
\begin{equation}
B_2:=\left\{(x_1, x_2):\begin{array}{cc}
(x_{1}-a)^2+(x_{2}-b)^{2} < (0.5)^{2}, \ \ a, b \in \{ -1.5, 0, 1.5 \}
    \end{array}
\right\}.  \label{6.8}
\end{equation}
In Figure \ref{true}, the closed blue curve is the boundary $\partial B_j$ of the support $B_j$, and the green brightness indicates the value of the true function on each cell divided into $(2M)^2$ in the sampling domain $[-S, S]^2$.
Here, we always employ the initial guess $q_0$ as
\begin{equation}
q_0\equiv0.  \label{6.9}
\end{equation}
\par
We demonstrate four algorithms: the Extended Kalman filter (\ref{EKF-1})--(\ref{EKF-B}) with initialization (\ref{EKF-B}) (EKF-initialization), Extended Kalman filter (\ref{EKF-1})--(\ref{EKF-B}) with update (\ref{update weight B}) (EKF-update),  Kalman filter Levenberg–Marquardt (\ref{KFL-1})--(\ref{KFL-3}) with initialization (\ref{KFL-B}) (KFL-initialization), and Kalman filter Levenberg–Marquardt (\ref{KFL-1})--(\ref{KFL-3}) with update (\ref{update weight B}) (KFL-initialization).
Figures \ref{B1}, \ref{B1noise}, \ref{B2}, and \ref{B2noise} show the reconstruction of $B_1$ and $B_2$ by the four algorithms. 
Figures \ref{B1}, \ref{B2} and \ref{B1noise}, \ref{B2noise} correspond to the noise-free and noise cases, respectively.
For the noise case, we add a random sampling from a normal distribution with a mean of zero and standard deviation $\sigma>0$ to our measurement $u^{\infty}(\cdot, \theta_{n})$, that is,
\begin{equation}
u^{\infty}(\cdot, \theta_{n})+\epsilon_{n}, \ \ \ \epsilon_{n} \sim \mathcal{N}(0, \sigma^{2}I).
\end{equation}
For the noise-free case in Figures \ref{B1} and \ref{B2}, we choose the regularization parameter $\alpha=100$ and $\alpha=3000$, respectively.
For the noise case in Figures \ref{B1noise} and \ref{B2noise}, we choose the standard deviation $\sigma=0.5$ and regularization parameters with $\alpha=2000$, and $\alpha=10000$, respectively.
The first and second rows correspond to a visualization of the four algorithms, and
the third row is the graph of the Mean Square Error (MSE) defined by
\begin{equation}
e_{i}:=\left\|q^{true}-q_{i} \right\|^2,  \label{6.11}
\end{equation}
where $q_{i}$ is associated with the state of the $i$-th iteration step. 
The horizontal and vertical axes correspond to the number of iterations and the MSE value, respectively.
Figures \ref{Graph sigma} and \ref{Graph alpha} show the MSE graph in the noise case for different standard deviations $\sigma$ and regularization parameters $\alpha$, respectively.
\par
We observe that the error of EKF (or initialization) decreases faster than that of KFL (or update) when we appropriately choose the regularization parameter; however, EKF (or initialization) is more sensitive to the choice of regularization parameter $\alpha$ and standard deviation $\sigma$ than KFL (or update), that is, EKF-initialization converges fastest and KFL-update is robustest to parameters.
\section*{Conclusions and future works}
In this paper, we proposed two reconstruction algorithms called the {\it Kalman Filter Levenberg-Marquardt} (KFL) and {\it iterative Extended Kalman Filter} (EKF) that are categorized as iterative optimization methods for the inverse medium scattering problem.
We numerically observed that the error of EKF decreases faster when the regularization parameter is appropriately chosen, whereas, EKF is more sensitive to the  regularization parameter and the noise than the KFL.
In addition, we observed that, by selecting the update of weight of the norm for each iteration step, both algorithms become robust to the parameters. 
KFL and EKF algorithms can be applied to various nonlinear inverse problems, and other applications will be developed in the future.
\section*{Acknowledgments}
The work of the first author was supported by Grant-in-Aid for JSPS Fellows (No.21J00119), Japan Society for the Promotion of Science.

\bibliographystyle{plain}
\bibliography{Inverse medium scattering problems with Kalman filter techniques.bbl}

\begin{thebibliography}{10}

\bibitem{alberti2020infinitedimensional}
G.~Alberti and M.~Santacesaria.
\newblock Infinite-dimensional inverse problems with finite measurements.
\newblock {\em Archive for Rational Mechanics and Analysis}, pages 1--31, 2021.

\bibitem{alessandrini2005determining}
G.~Alessandrini and L.~Rondi.
\newblock Determining a sound-soft polyhedral scatterer by a single far-field
  measurement.
\newblock {\em Proceedings of the American Mathematical Society},
  133(6):1685--1691, 2005.

\bibitem{arridge_maass_oktem_schonlieb_2019}
S.~Arridge, P.~Maass, O.~Öktem, and C.~B. Schönlieb.
\newblock Solving inverse problems using data-driven models.
\newblock {\em Acta Numerica}, 28:1–174, 2019.

\bibitem{Bakushinsky}
A.~B. Bakushinsky and M.~Yu. Kokurin.
\newblock {\em Iterative methods for approximate solution of inverse problems},
  volume 577.
\newblock Springer Science \& Business Media, 2005.

\bibitem{Bao2005InverseMS}
G.~Bao and P.~Li.
\newblock Inverse medium scattering for the helmholtz equation at fixed
  frequency.
\newblock {\em Inverse Problems}, 21:1621--1641, 2005.

\bibitem{Bao}
G.~Bao and F.~Triki.
\newblock Error estimates for the recursive linearization of inverse medium
  problems.
\newblock {\em Journal of Computational Mathematics}, pages 725--744, 2010.

\bibitem{BOURGEOIS2013187}
L.~Bourgeois.
\newblock A remark on lipschitz stability for inverse problems.
\newblock {\em Comptes Rendus Mathematique}, 351(5):187--190, 2013.

\bibitem{bukhgeim2008recovering}
A.~L. Bukhgeim.
\newblock Recovering a potential from cauchy data in the two-dimensional case.
\newblock 2008.

\bibitem{Cakoni}
F.~Cakoni and D.~Colton.
\newblock {\em Qualitative methods in inverse scattering theory: An
  introduction}.
\newblock Springer Science \& Business Media, 2005.

\bibitem{Chen}
X.~Chen.
\newblock {\em Computational methods for electromagnetic inverse scattering}.
\newblock John Wiley \& Sons, 2018.

\bibitem{cheng2003uniqueness}
J.~Cheng and M.~Yamamoto.
\newblock Uniqueness in an inverse scattering problem within non-trapping
  polygonal obstacles with at most two incoming waves.
\newblock {\em Inverse Problems}, 19(6):1361, 2003.

\bibitem{ColtonKirsch}
D.~Colton and A.~Kirsch.
\newblock A simple method for solving inverse scattering problems in the
  resonance region.
\newblock {\em Inverse problems}, 12(4):383, 1996.

\bibitem{ColtonKress}
D.~Colton and R.~Kress.
\newblock {\em Inverse Acoustic and Electromagnetic Scattering Theory},
  volume~93.
\newblock Springer Nature, 2019.

\bibitem{Freitag}
M.~A. Freitag and R.~W.~E. Potthast.
\newblock Synergy of inverse problems and data assimilation techniques.
\newblock In {\em Large scale inverse problems}, pages 1--54. De Gruyter, 2013.

\bibitem{Giorgi}
G.~Giorgi, M.~Brignone, R.~Aramini, and M.~Piana.
\newblock Application of the inhomogeneous lippmann--schwinger equation to
  inverse scattering problems.
\newblock {\em SIAM Journal on Applied Mathematics}, 73(1):212--231, 2013.

\bibitem{grewal2010applications}
M.~S. Grewal and A.~P. Andrews.
\newblock Applications of kalman filtering in aerospace 1960 to the present
  [historical perspectives].
\newblock {\em IEEE Control Systems Magazine}, 30(3):69--78, 2010.

\bibitem{Grewal}
M.~S. Grewal and A.~P. Andrews.
\newblock {\em Kalman filtering: Theory and Practice with MATLAB}.
\newblock John Wiley \& Sons, 2014.

\bibitem{Hanke_1997}
M.~Hanke.
\newblock A regularizing levenberg - marquardt scheme, with applications to
  inverse groundwater filtration problems.
\newblock {\em Inverse Problems}, 13(1):79--95, feb 1997.

\bibitem{Hanke}
M.~Hanke.
\newblock {\em A Taste of Inverse Problems: Basic Theory and Examples}.
\newblock SIAM, 2017.

\bibitem{Hohage}
T.~Hohage.
\newblock On the numerical solution of a three-dimensional inverse medium
  scattering problem.
\newblock {\em Inverse Problems}, 17(6):1743, 2001.

\bibitem{Honda}
N.~Honda, G.~Nakamura, R.~Potthast, and M.~Sini.
\newblock The no-response approach and its relation to non-iterative methods
  for the inverse scattering.
\newblock {\em Annali di Matematica Pura ed Applicata}, 187(1):7--37, 2008.

\bibitem{Ikehata}
M.~Ikehata.
\newblock Reconstruction of an obstacle from the scattering amplitude at a
  fixed frequency.
\newblock {\em Inverse Problems}, 14(4):949--954, aug 1998.

\bibitem{ito2012direct}
K.~Ito, B.~Jin, and J.~Zou.
\newblock A direct sampling method to an inverse medium scattering problem.
\newblock {\em Inverse Problems}, 28(2):025003, 2012.

\bibitem{Jazwinski}
A.~H. Jazwinski.
\newblock {\em Stochastic processes and filtering theory}.
\newblock Courier Corporation, 2007.

\bibitem{Kalman}
R.~E. Kalman.
\newblock A new approach to linear filtering and prediction problems. trans.
  asme. ser.
\newblock {\em D: J. Basic Eng.}, 82(1960), 1960.

\bibitem{Kaltenbacher}
B.~Kaltenbacher, A.~Neubauer, and O.~Scherzer.
\newblock {\em Iterative regularization methods for nonlinear ill-posed
  problems}.
\newblock de Gruyter, 2008.

\bibitem{Kirsch}
A.~Kirsch.
\newblock {\em An introduction to the mathematical theory of inverse problems},
  volume 120.
\newblock Springer, 2011.

\bibitem{Kirsch2}
A.~Kirsch.
\newblock Remarks on the born approximation and the factorization method.
\newblock {\em Applicable Analysis}, 96(1):70--84, 2017.

\bibitem{KirschGrinberg}
A.~Kirsch and N.~Grinberg.
\newblock {\em The factorization method for inverse problems}.
\newblock Number~36. Oxford University Press, 2008.

\bibitem{Kress}
R.~Kress.
\newblock {\em Linear Integral Equations}, volume~82.
\newblock Springer Science \& Business Media, 2013.

\bibitem{Liu_2018}
Juan L. and Jiguang S.
\newblock Extended sampling method in inverse scattering.
\newblock {\em Inverse Problems}, 34(8):085007, jun 2018.

\bibitem{liu1997inverse}
C.~Liu.
\newblock Inverse obstacle problem: local uniqueness for rougher obstacles and
  the identification of a ball.
\newblock {\em Inverse Problems}, 13(4):1063, 1997.

\bibitem{liu2006uniqueness}
H.~Liu and J.~Zou.
\newblock Uniqueness in an inverse acoustic obstacle scattering problem for
  both sound-hard and sound-soft polyhedral scatterers.
\newblock {\em Inverse Problems}, 22(2):515, 2006.

\bibitem{NakamuraPotthast}
G.~Nakamura and R.~Potthast.
\newblock {\em Inverse modeling}.
\newblock IOP Publishing, 2015.

\bibitem{novikov1988multidimensional}
R.~G. Novikov.
\newblock Multidimensional inverse spectral problem for the equation $- \delta
  \psi +(v(x)- \mathrm{E} u(x))\psi= 0$.
\newblock {\em Functional Analysis and Its Applications}, 22(4):263--272, 1988.

\bibitem{Potthast}
R.~Potthast.
\newblock {\em Point sources and multipoles in inverse scattering theory}.
\newblock Chapman and Hall/CRC, 2001.

\bibitem{ramm1988recovery}
A.~G. Ramm.
\newblock Recovery of the potential from fixed-energy scattering data.
\newblock {\em Inverse problems}, 4(3):877, 1988.

\bibitem{Pike}
P.~C. Sabatier and E.~R. Pike.
\newblock {\em Scattering: scattering and inverse scattering in pure and
  applied science}.
\newblock Academic Press, 2002.

\bibitem{saranen2001periodic}
J.~Saranen and G.~Vainikko.
\newblock {\em Periodic integral and pseudodifferential equations with
  numerical approximation}.
\newblock Springer Science \& Business Media, 2001.

\bibitem{vainikko2000fast}
G.~Vainikko.
\newblock Fast solvers of the lippmann-schwinger equation.
\newblock In {\em Direct and inverse problems of mathematical physics}, pages
  423--440. Springer, 2000.

\end{thebibliography}

\begin{figure}[h]
\begin{minipage}[b]{0.5\linewidth}
\centering
\includegraphics[keepaspectratio, scale=0.5]{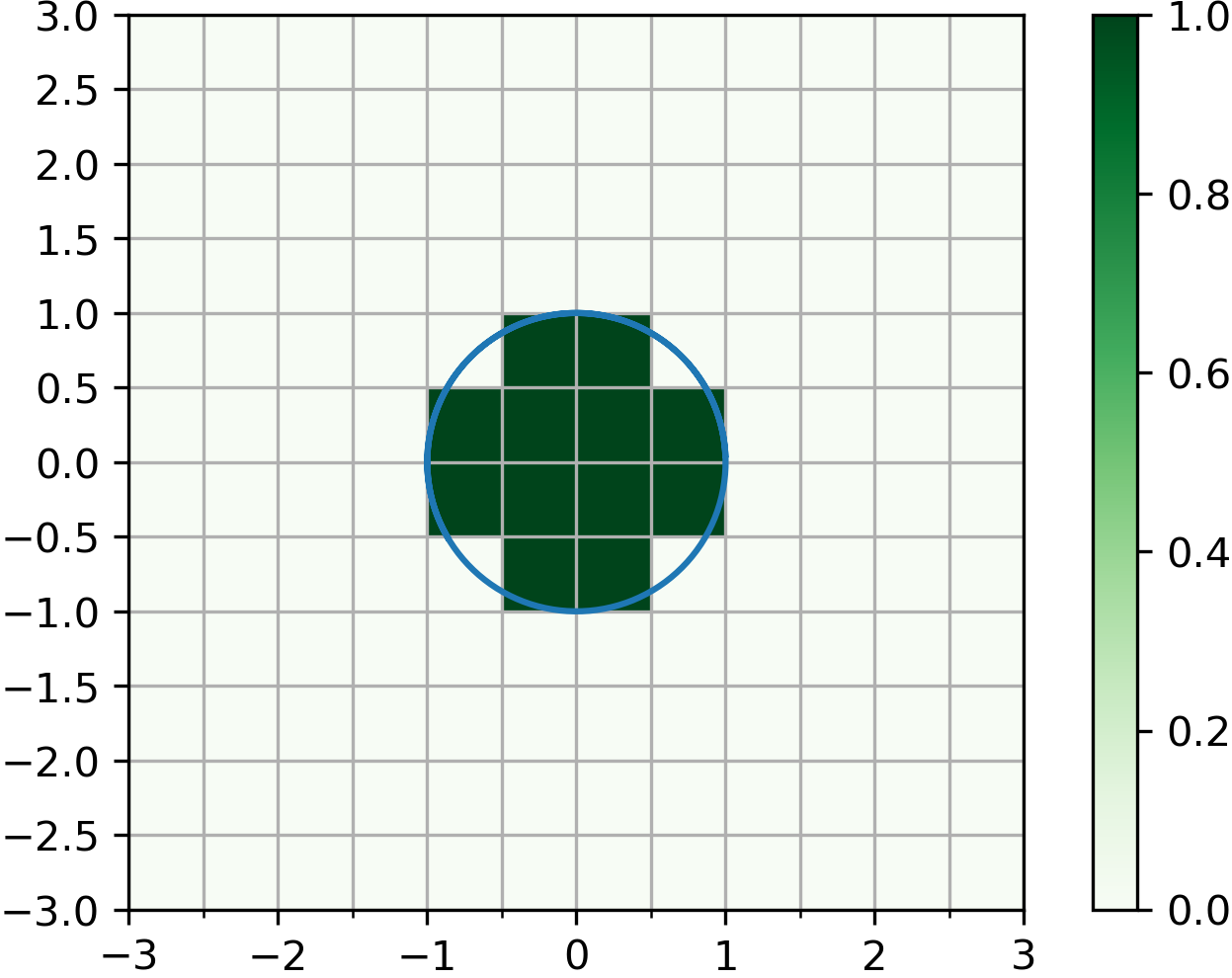}
\subcaption{$q^{true}_{1}$}
\end{minipage}
\begin{minipage}[b]{0.5\linewidth}
\centering
\includegraphics[keepaspectratio, scale=0.5]{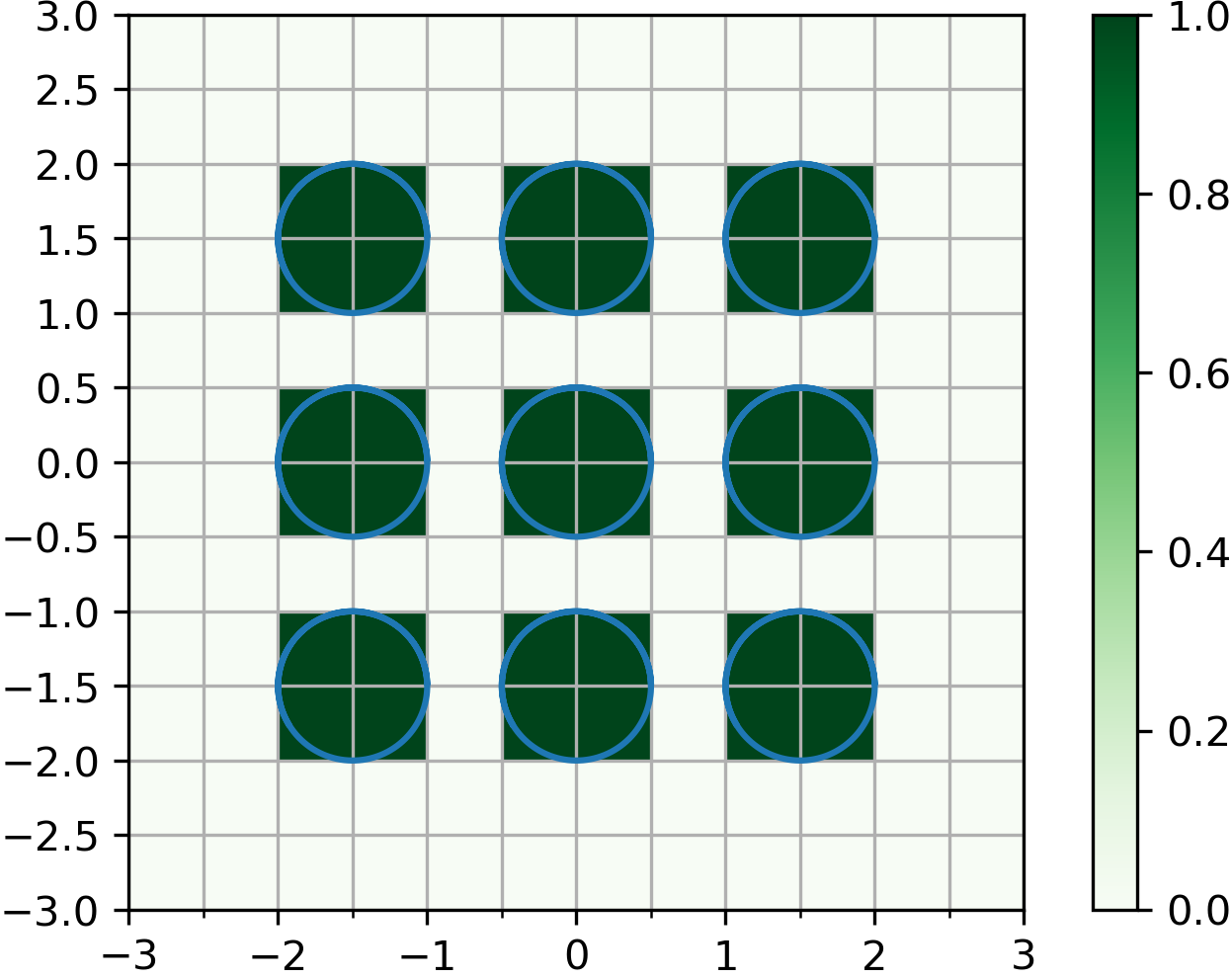}
\subcaption{$q^{true}_{2}$}
\end{minipage}
\caption{True function}
\label{true}
\end{figure}

\begin{figure}[h]
\begin{tabular}{c}
\begin{minipage}[b]{0.5\linewidth}
\centering
\includegraphics[keepaspectratio, scale=0.5]{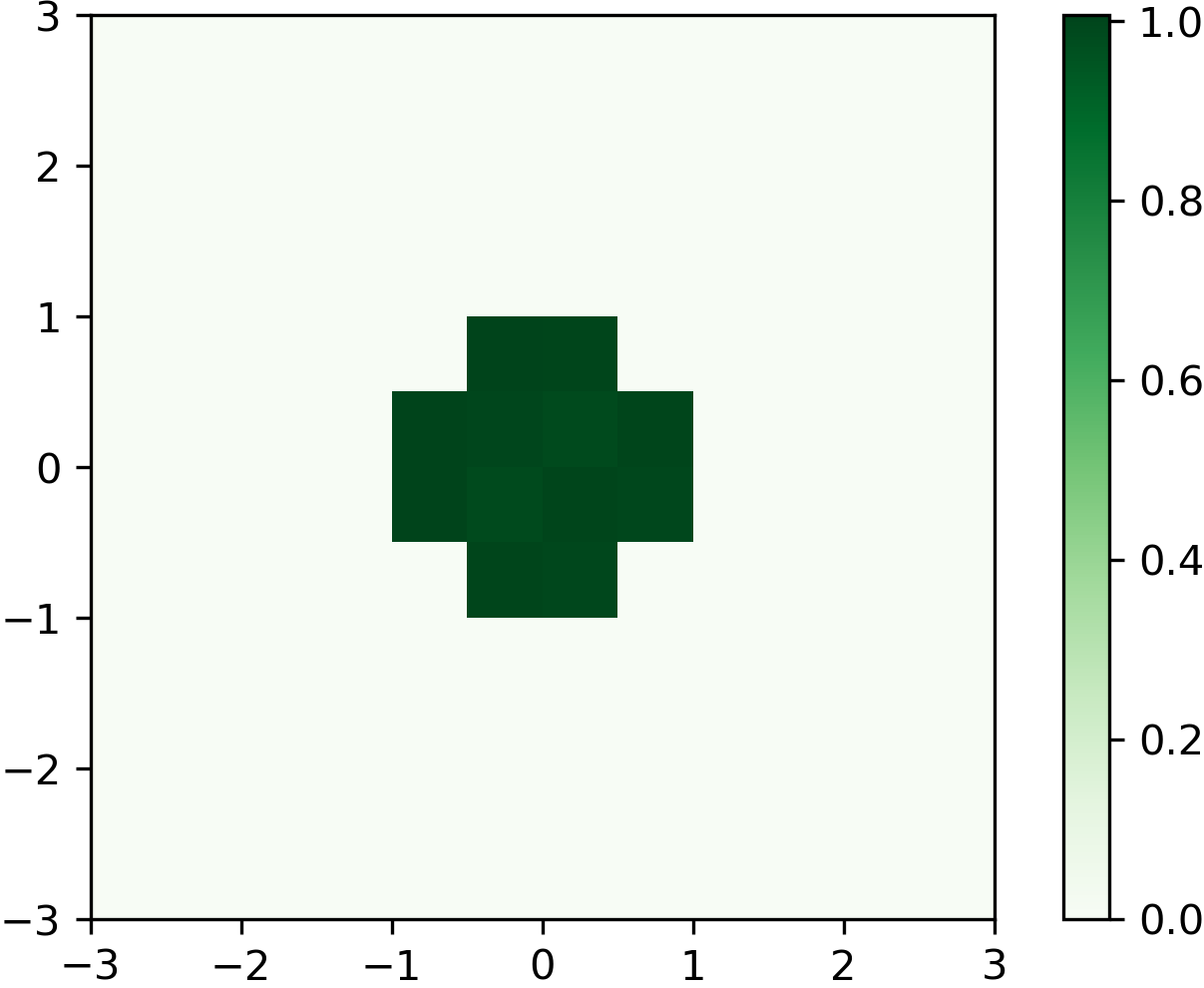}
\subcaption{EKF-initialization}
\end{minipage}
\begin{minipage}[b]{0.5\linewidth}
\centering
\includegraphics[keepaspectratio, scale=0.5]{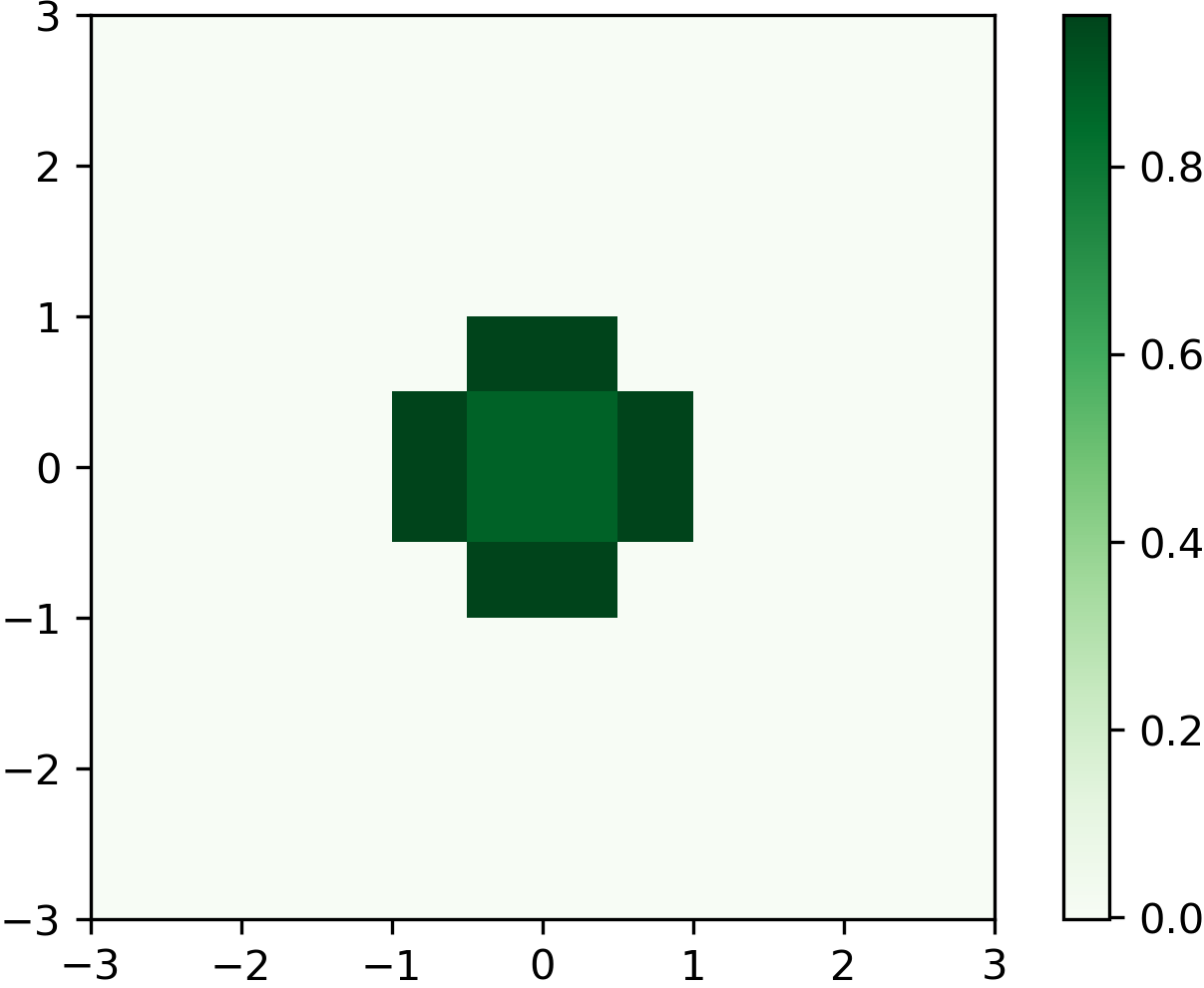}
\subcaption{KFL-initialization}
\end{minipage}
\end{tabular}
\begin{tabular}{c}
\begin{minipage}[b]{0.5\linewidth}
\centering
\includegraphics[keepaspectratio, scale=0.5]{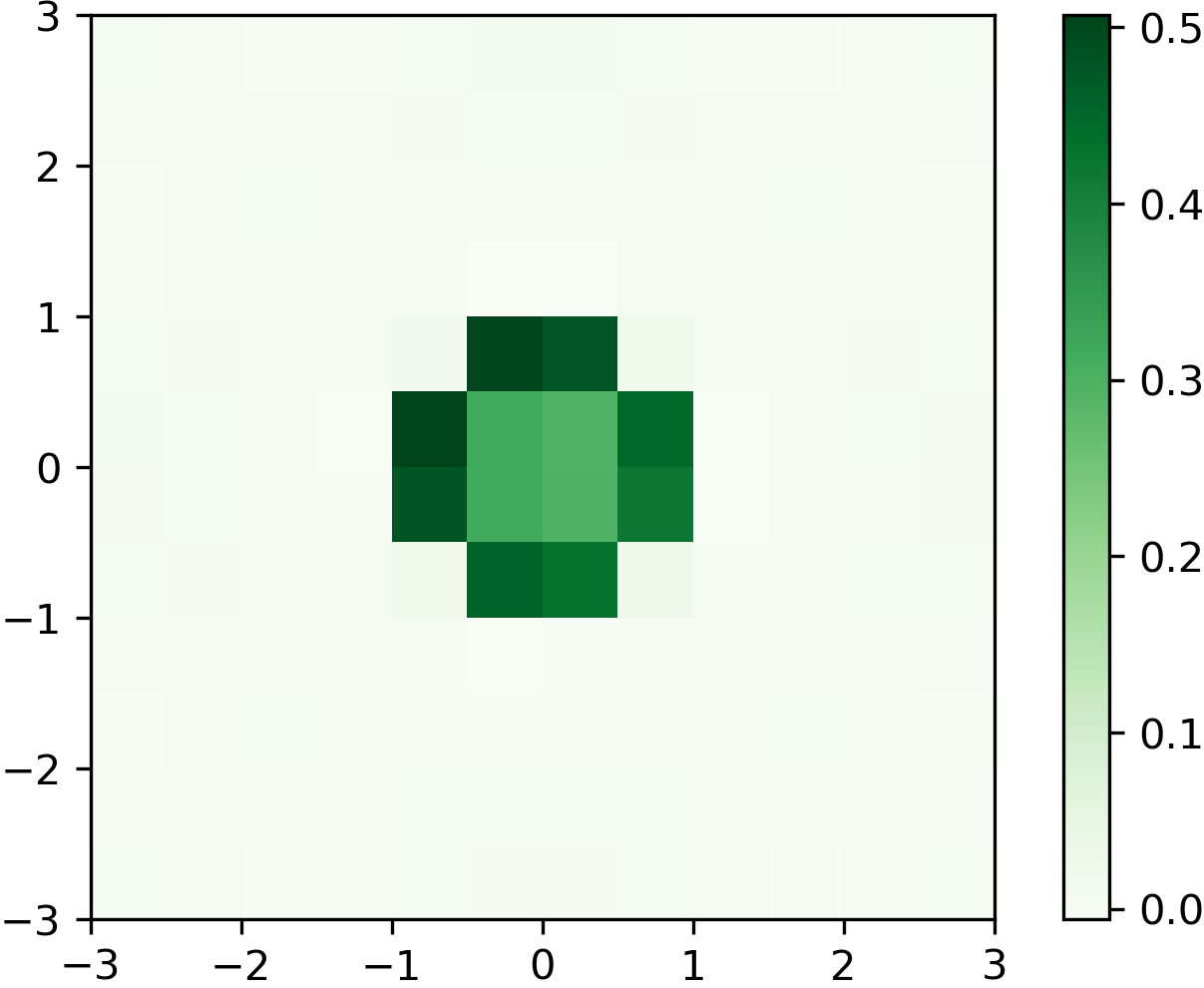}
\subcaption{EKF-update}
\end{minipage}
\begin{minipage}[b]{0.5\linewidth}
\centering
\includegraphics[keepaspectratio, scale=0.5]{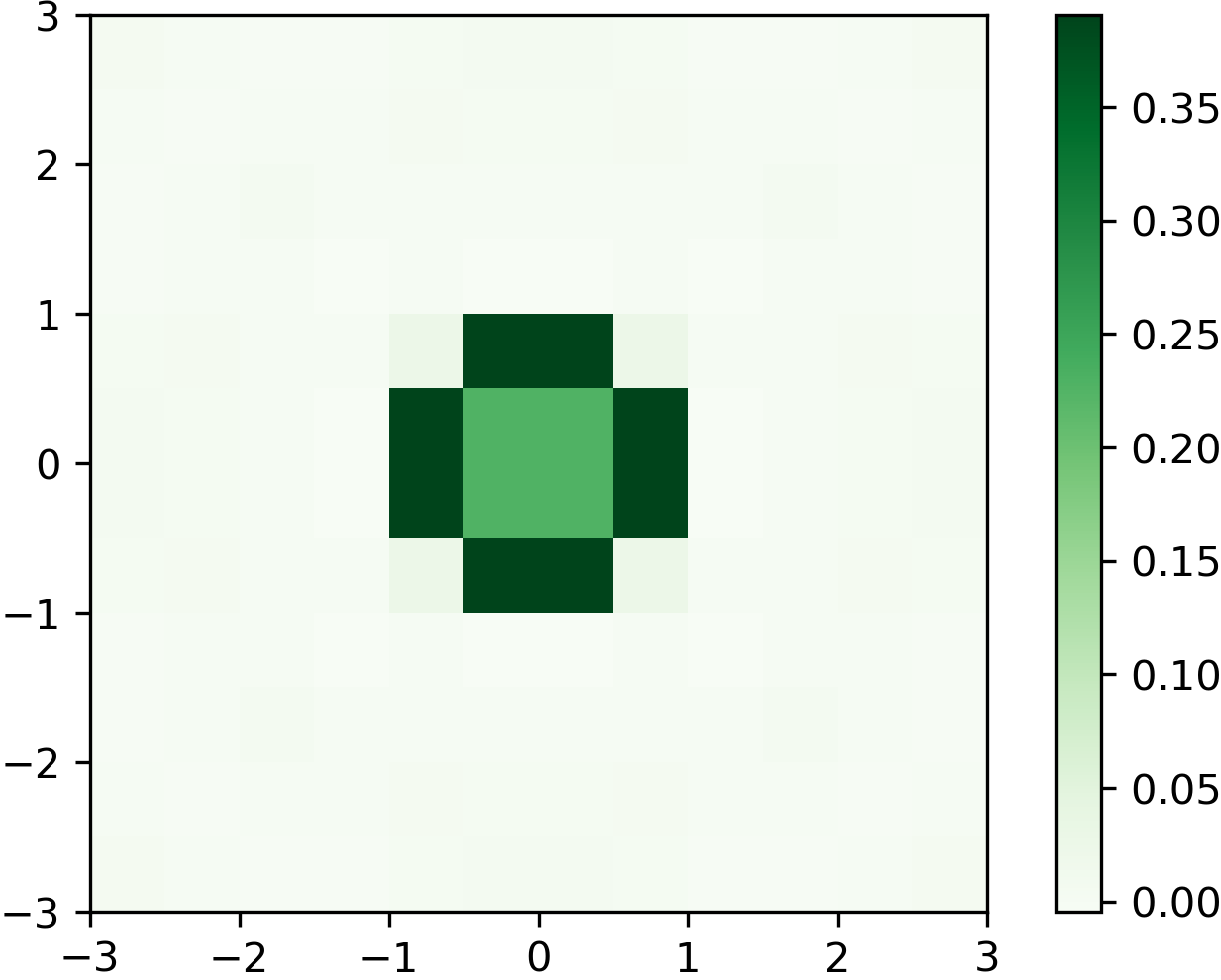}
\subcaption{KFL-update}
\end{minipage}
\vspace{1cm}
\end{tabular}
\begin{tabular}{c}
\hspace{2cm}
\begin{minipage}[b]{0.6\linewidth}
\centering
\includegraphics[keepaspectratio, scale=0.67]{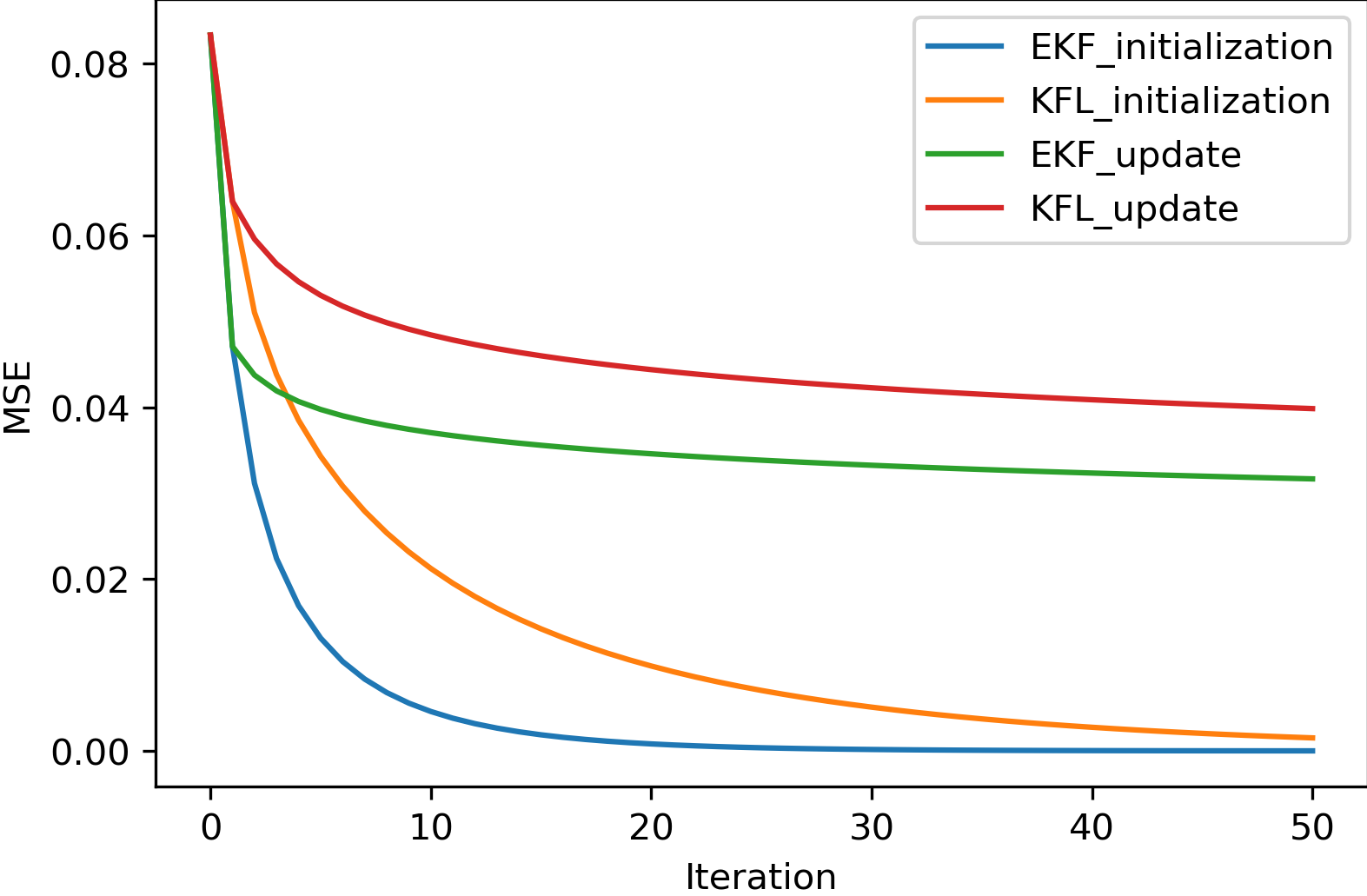}
\end{minipage}
\end{tabular}
\caption{Reconstruction for the noise-free case: $q_{1}^{true}$, $\alpha=100$}
\label{B1}
\end{figure}

\begin{figure}[h]
\begin{tabular}{c}
\begin{minipage}[b]{0.5\linewidth}
\centering
\includegraphics[keepaspectratio, scale=0.5]{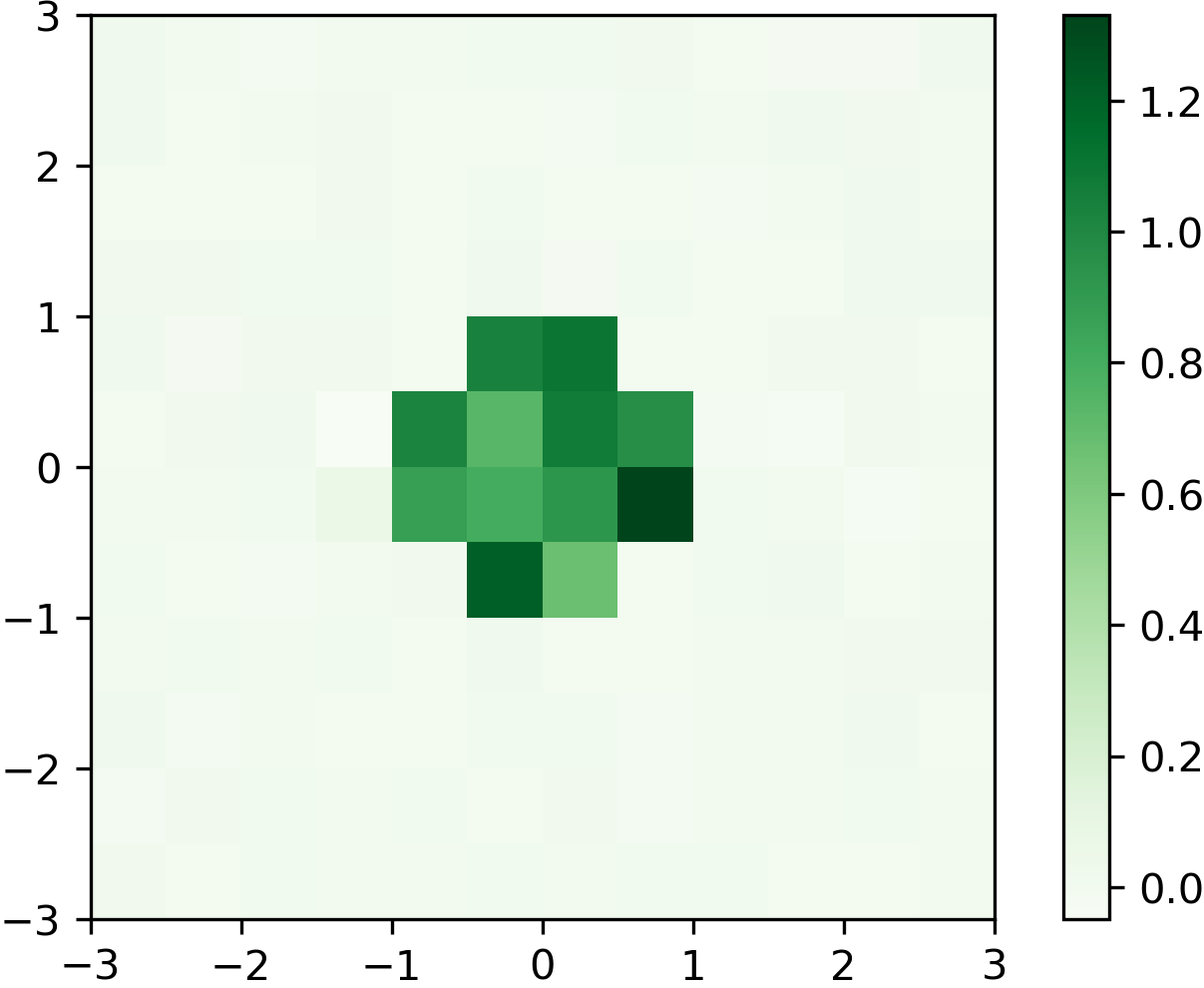}
\subcaption{EKF-initialization}
\end{minipage}
\begin{minipage}[b]{0.5\linewidth}
\centering
\includegraphics[keepaspectratio, scale=0.5]{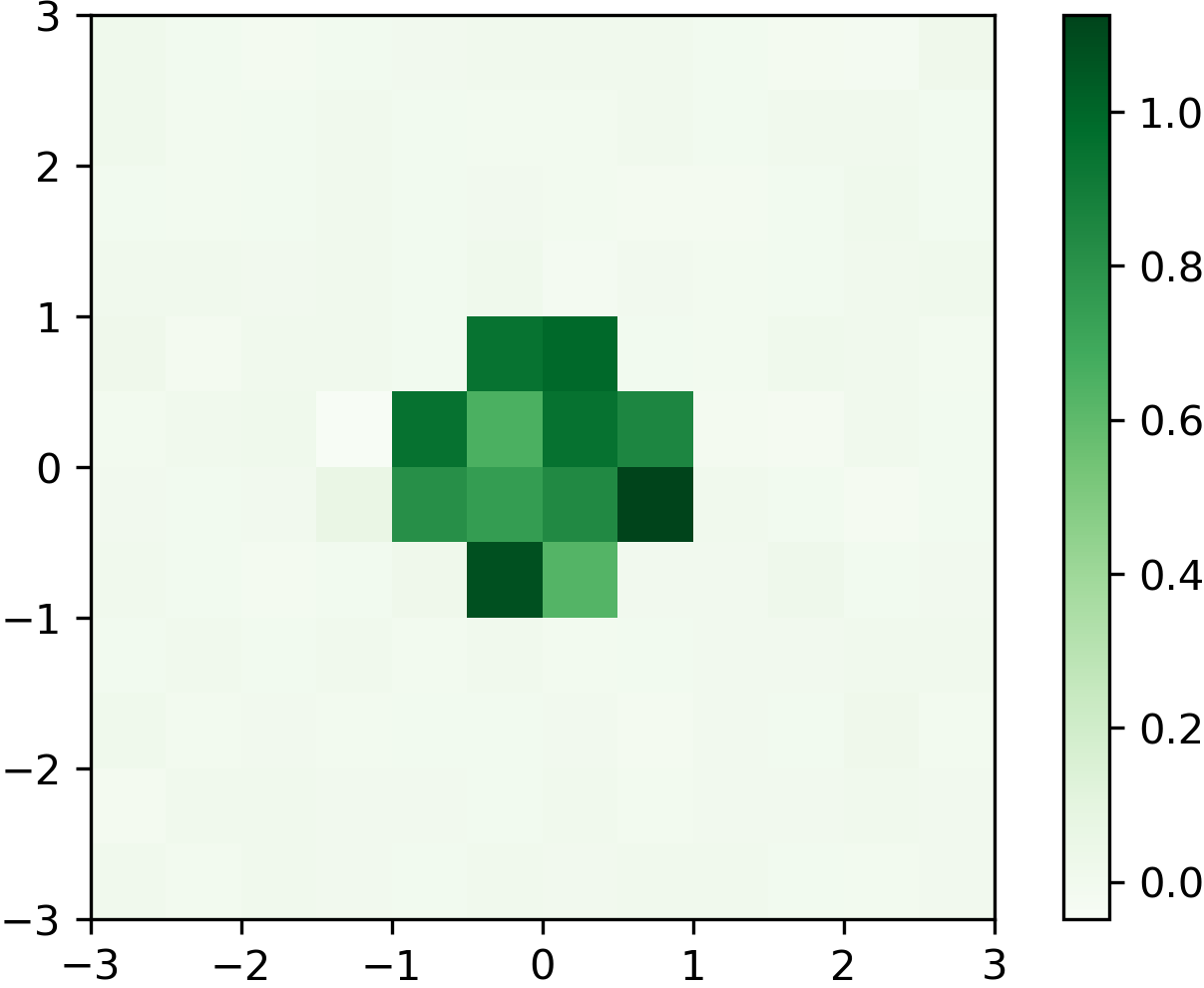}
\subcaption{KFL-initialization}
\end{minipage}
\end{tabular}
\begin{tabular}{c}
\begin{minipage}[b]{0.5\linewidth}
\centering
\includegraphics[keepaspectratio, scale=0.5]{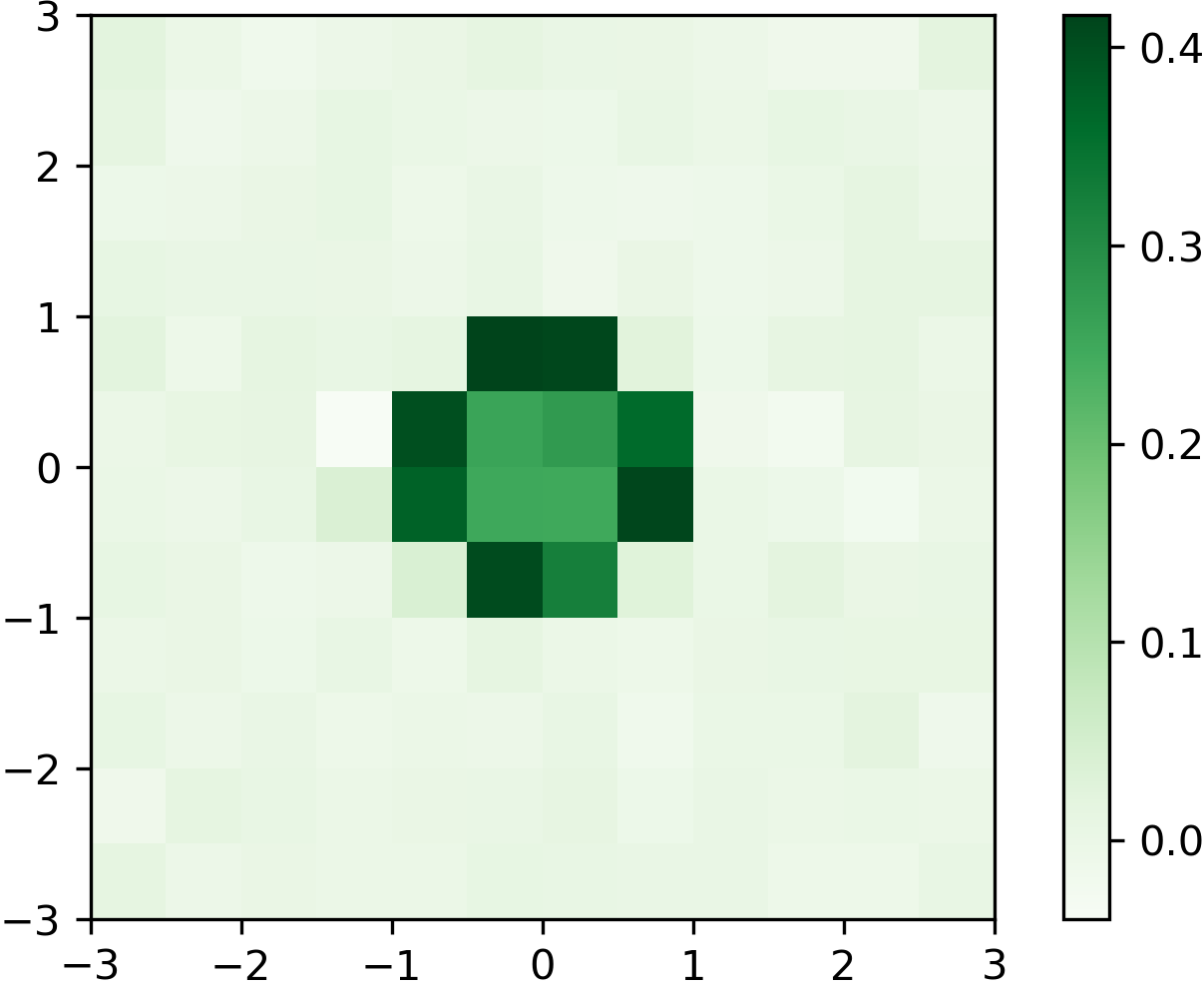}
\subcaption{EKF-update}
\end{minipage}
\begin{minipage}[b]{0.5\linewidth}
\centering
\includegraphics[keepaspectratio, scale=0.5]{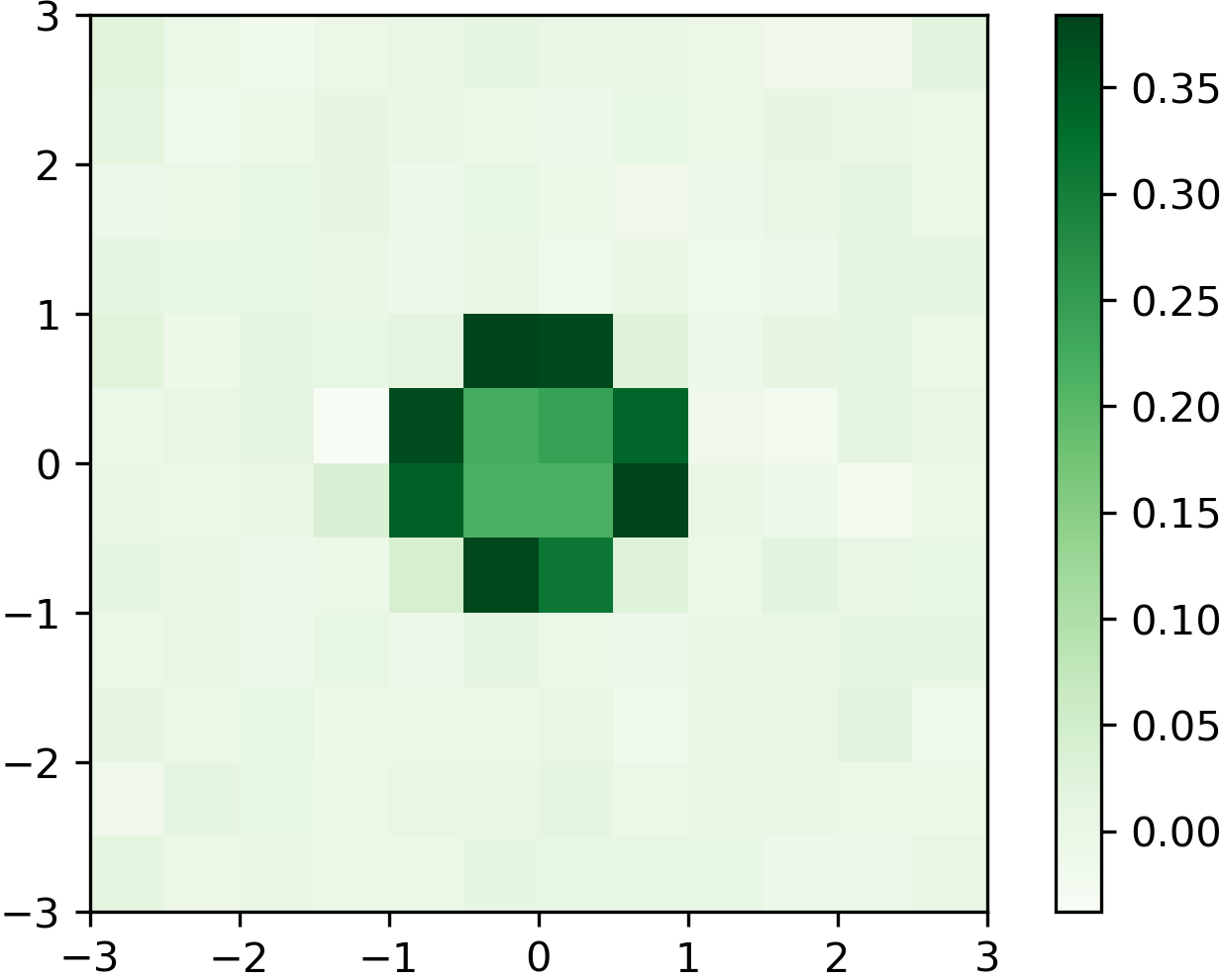}
\subcaption{KFL-update}
\end{minipage}
\vspace{1cm}
\end{tabular}
\begin{tabular}{c}
\hspace{2cm}
\begin{minipage}[b]{0.6\linewidth}
\centering
\includegraphics[keepaspectratio, scale=0.67]{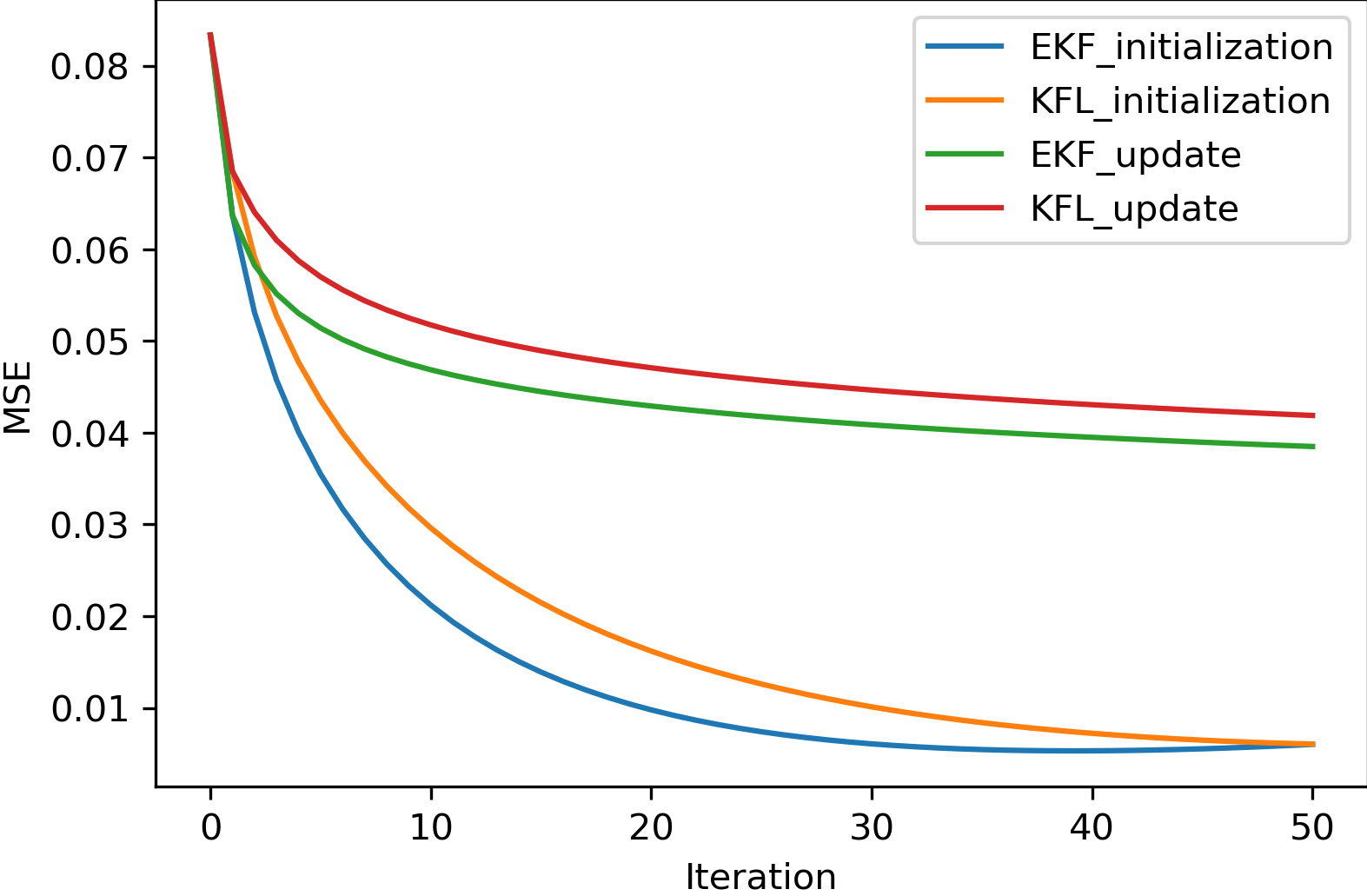}
\end{minipage}
\end{tabular}
\caption{Reconstruction for the noise case: $q_{1}^{true}$, $\alpha=2000$, $\sigma=0.5$}
\label{B1noise}
\end{figure}

\begin{figure}[h]
\begin{tabular}{c}
\begin{minipage}[b]{0.5\linewidth}
\centering
\includegraphics[keepaspectratio, scale=0.5]{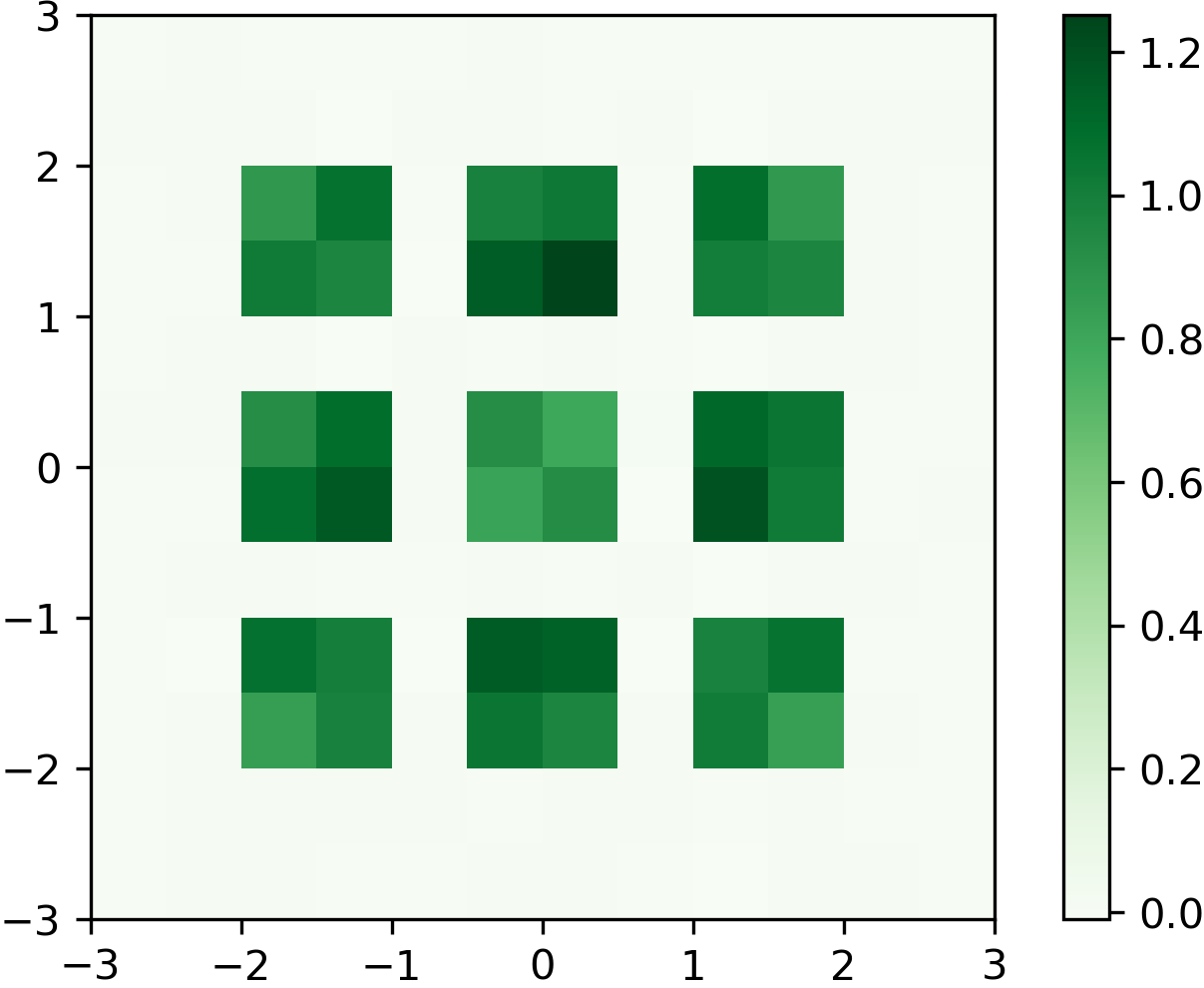}
\subcaption{EKF-initialization}
\end{minipage}
\begin{minipage}[b]{0.5\linewidth}
\centering
\includegraphics[keepaspectratio, scale=0.5]{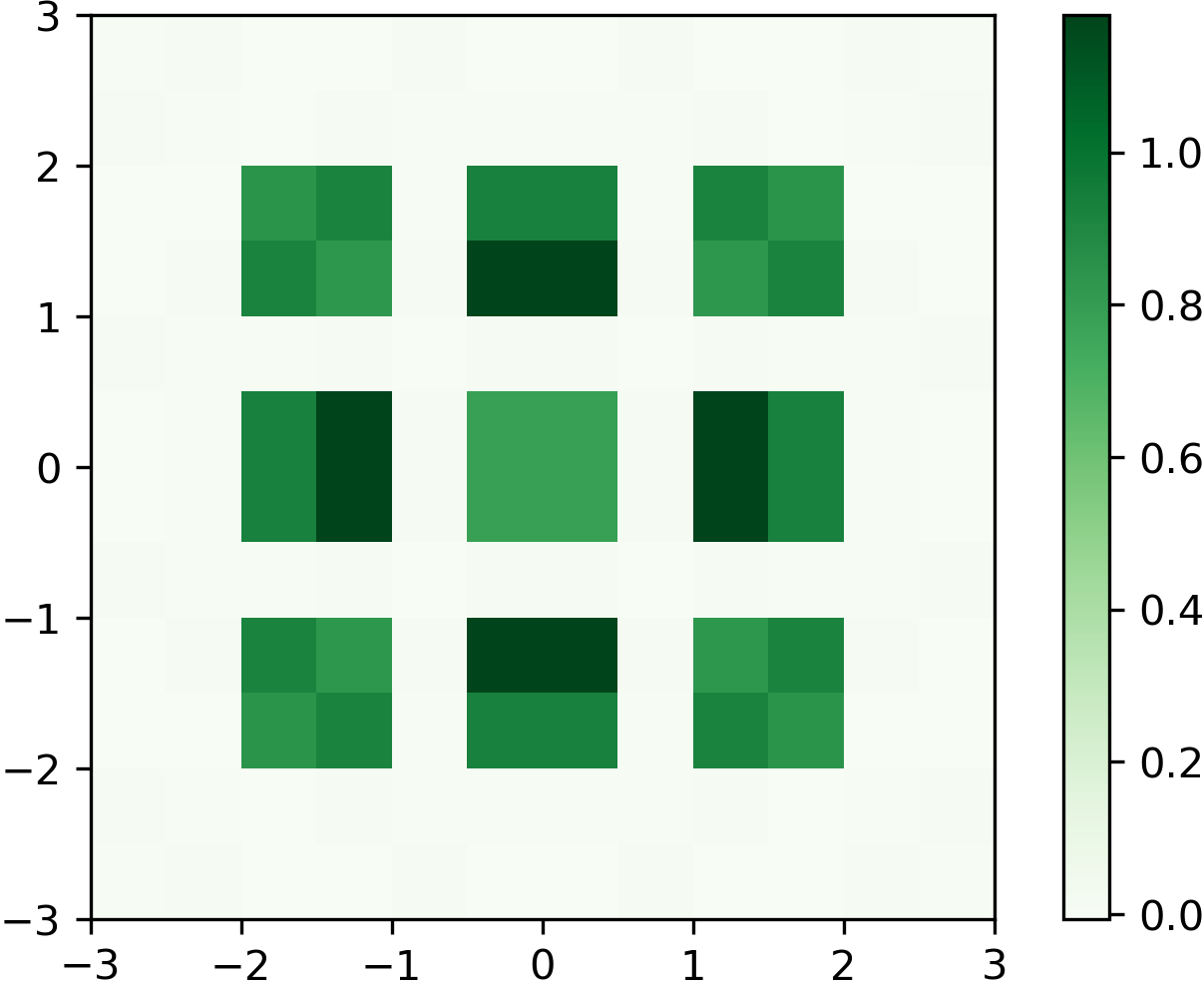}
\subcaption{KFL-initialization}
\end{minipage}
\end{tabular}
\begin{tabular}{c}
\begin{minipage}[b]{0.5\linewidth}
\centering
\includegraphics[keepaspectratio, scale=0.5]{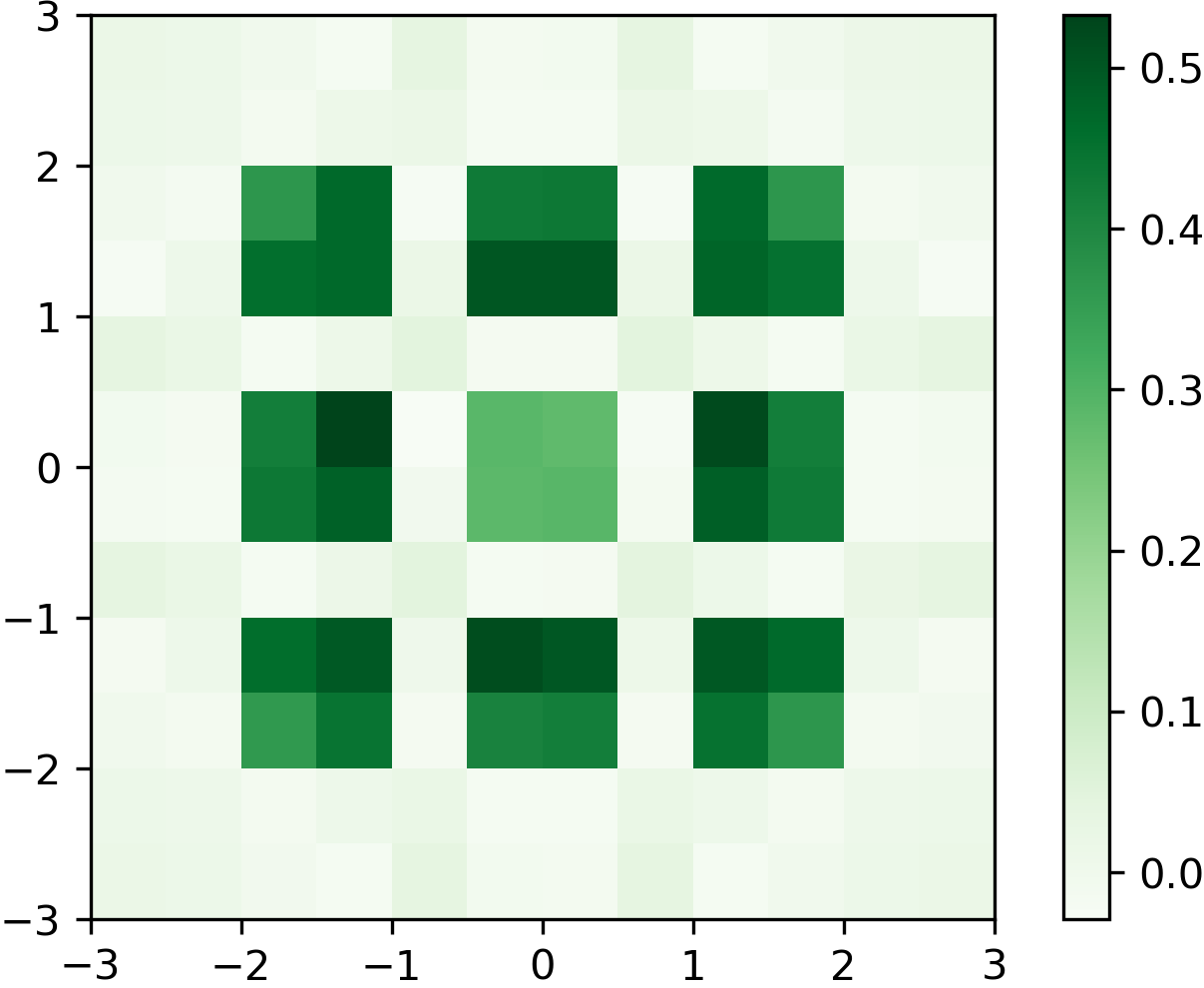}
\subcaption{EKF-update}
\end{minipage}
\begin{minipage}[b]{0.5\linewidth}
\centering
\includegraphics[keepaspectratio, scale=0.5]{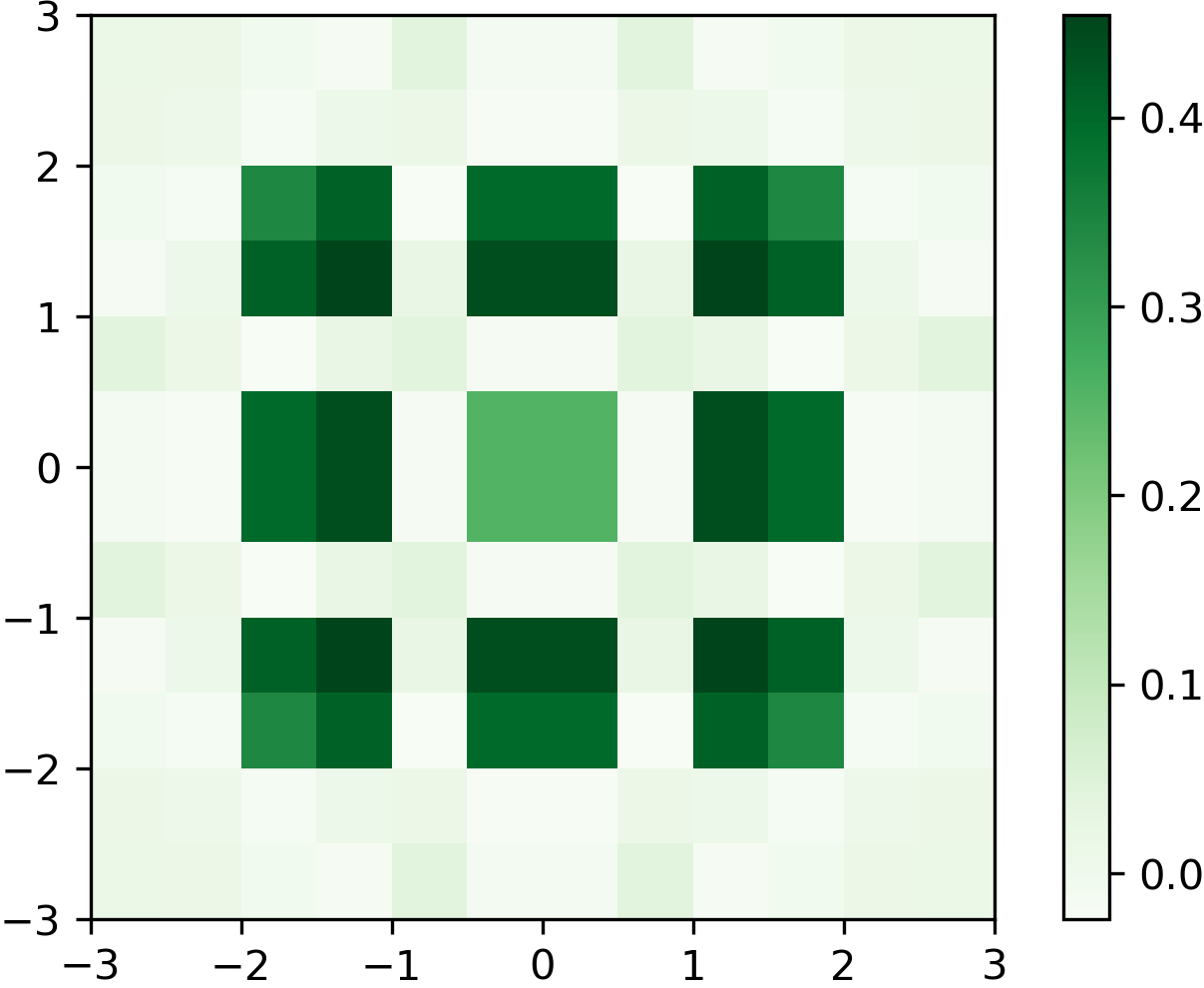}
\subcaption{KFL-update}
\end{minipage}
\vspace{1cm}
\end{tabular}
\begin{tabular}{c}
\hspace{2cm}
\begin{minipage}[b]{0.6\linewidth}
\centering
\includegraphics[keepaspectratio, scale=0.67]{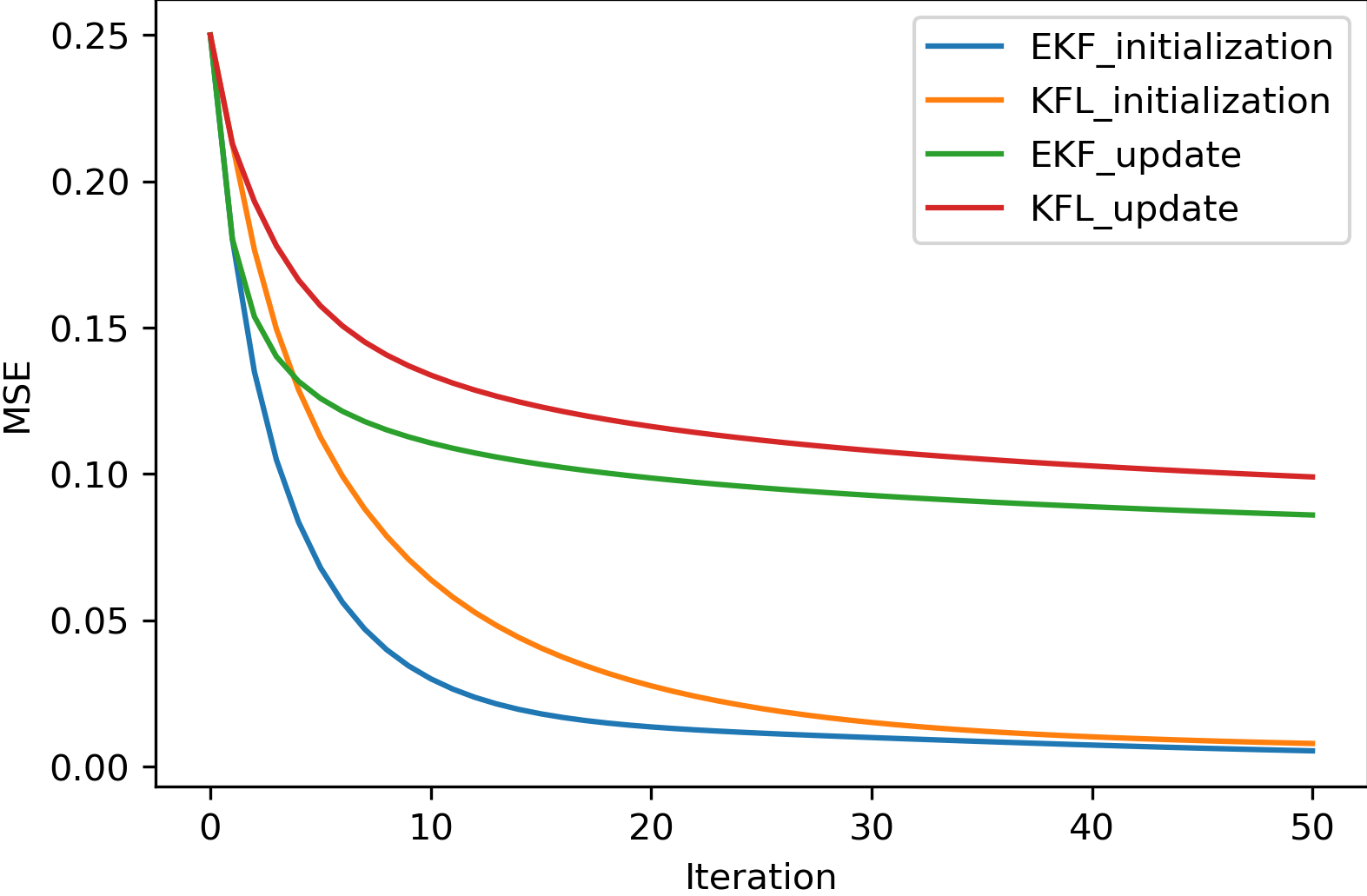}
\end{minipage}
\end{tabular}
\caption{Reconstruction for the noise-free case: $q_{2}^{true}$, $\alpha=3000$}
\label{B2}
\end{figure}

\begin{figure}[h]
\begin{tabular}{c}
\begin{minipage}[b]{0.5\linewidth}
\centering
\includegraphics[keepaspectratio, scale=0.5]{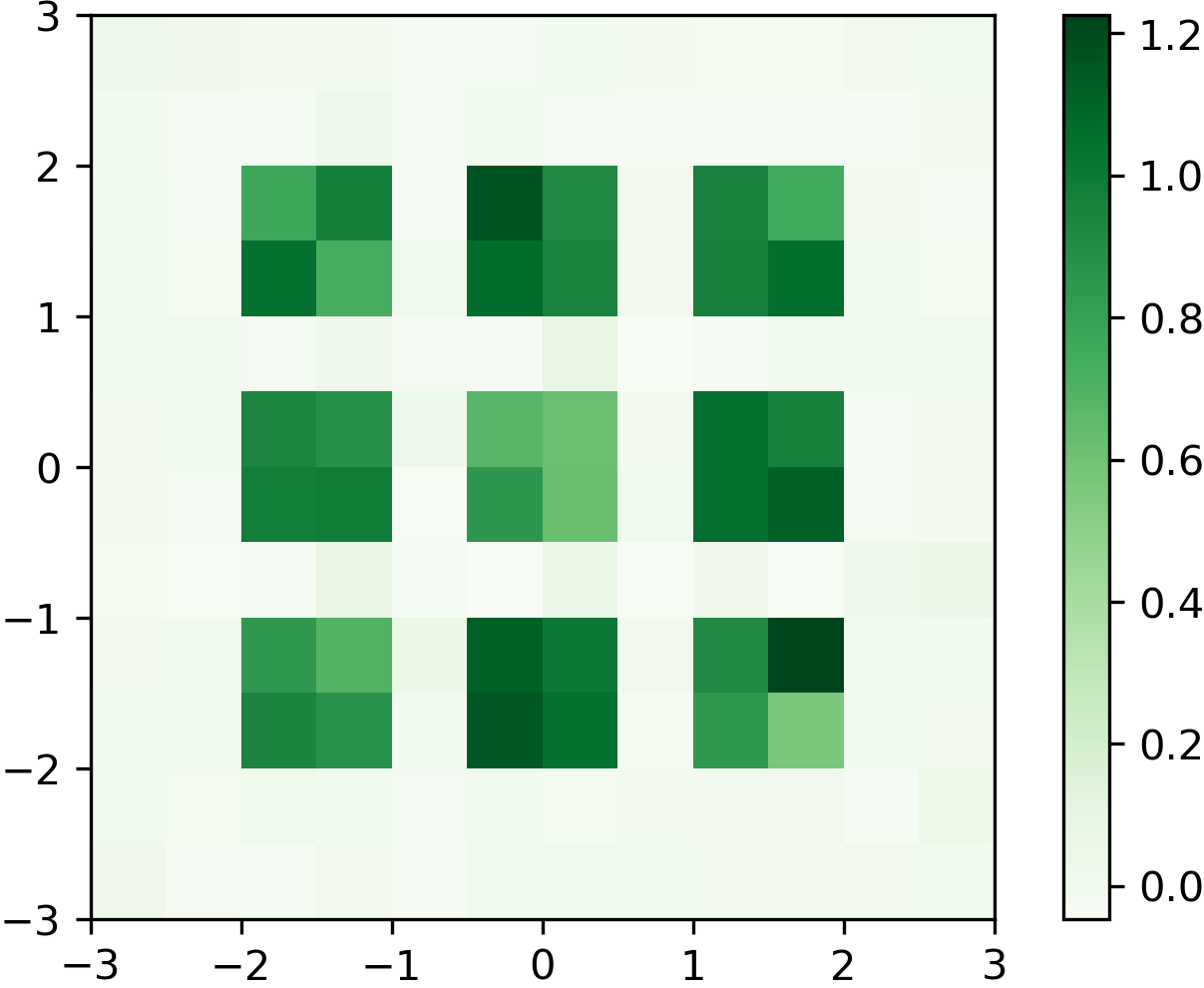}
\subcaption{EKF-initialization}
\end{minipage}
\begin{minipage}[b]{0.5\linewidth}
\centering
\includegraphics[keepaspectratio, scale=0.5]{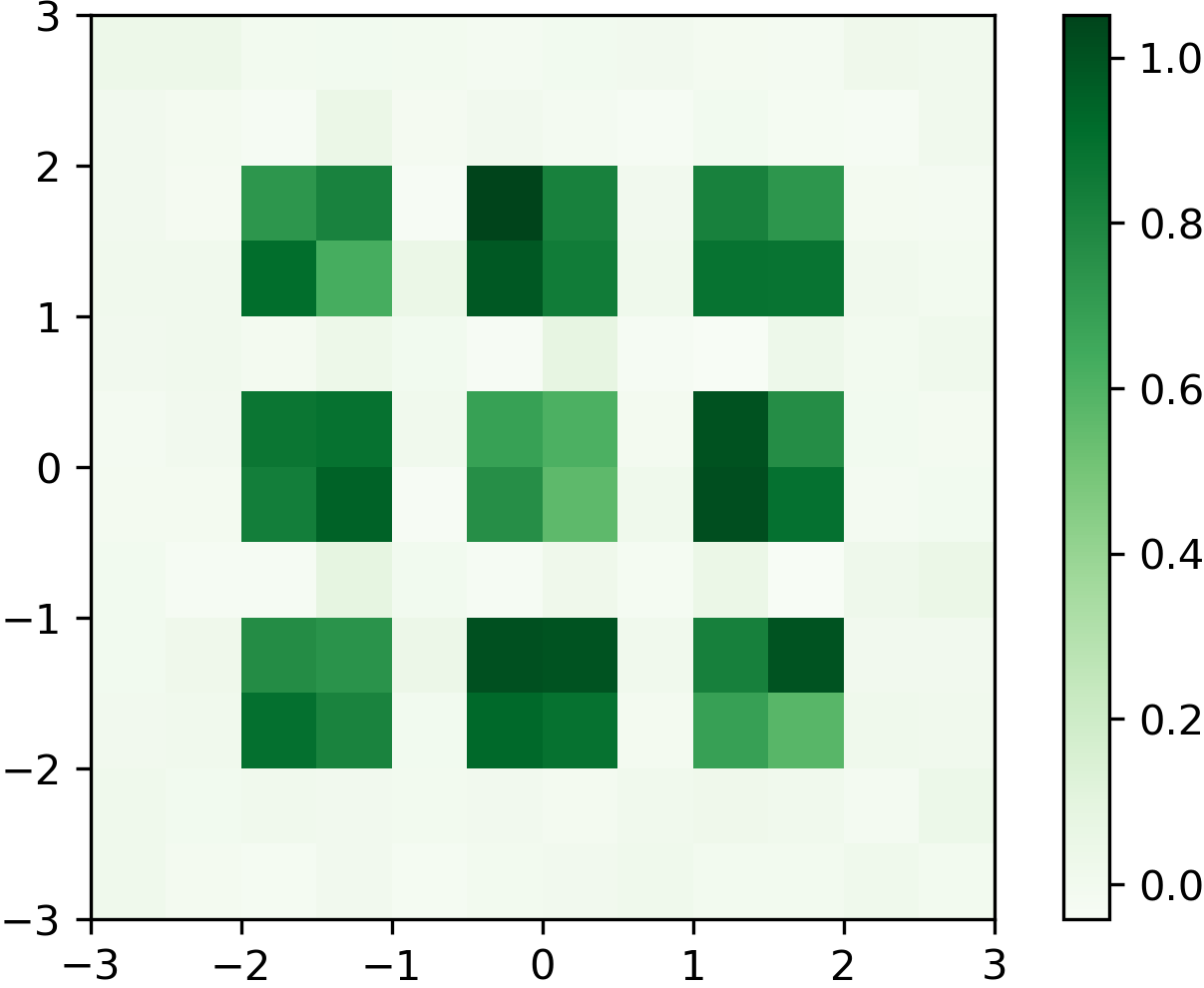}
\subcaption{KFL-initialization}
\end{minipage}
\end{tabular}
\begin{tabular}{c}
\begin{minipage}[b]{0.5\linewidth}
\centering
\includegraphics[keepaspectratio, scale=0.5]{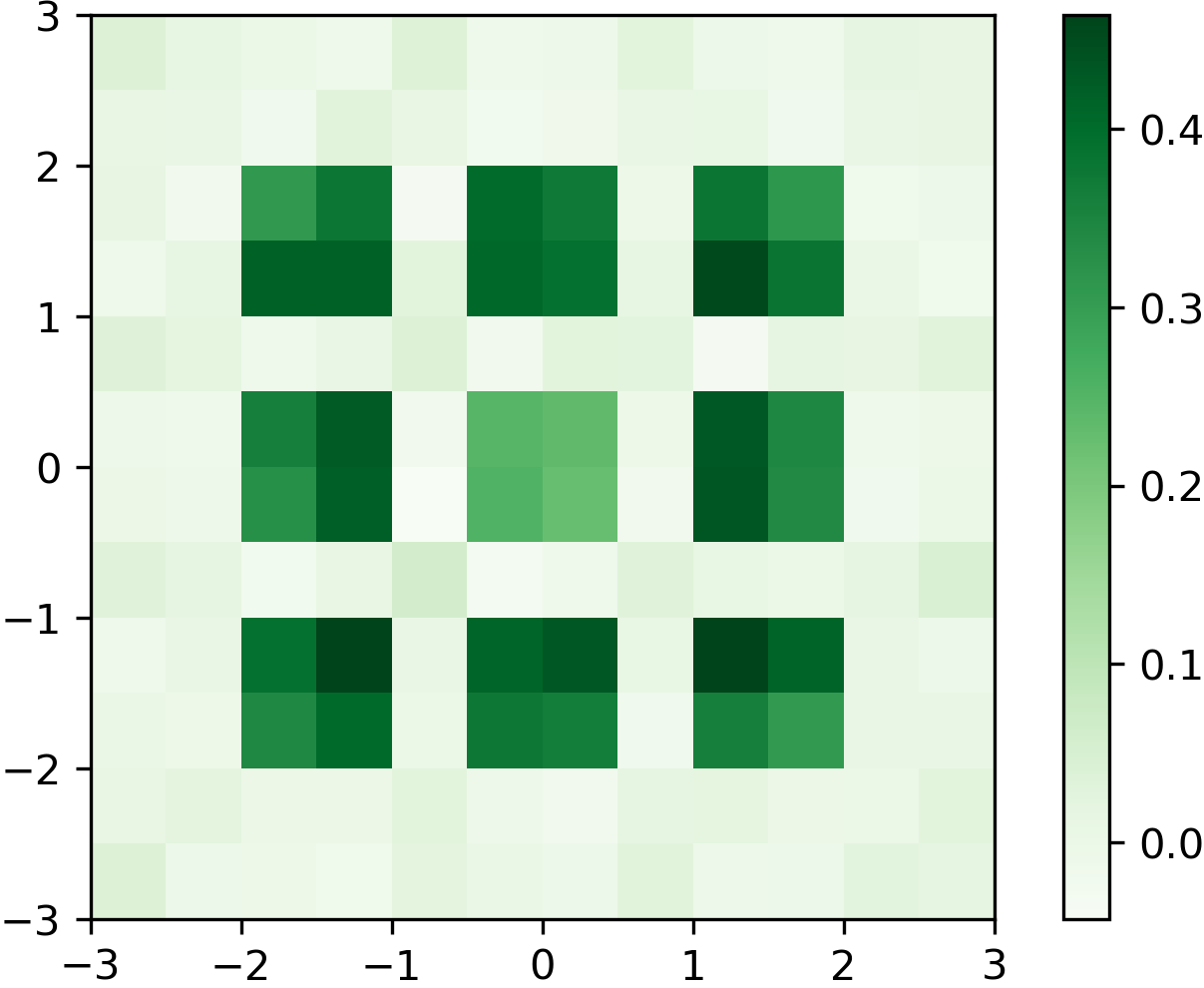}
\subcaption{EKF-update}
\end{minipage}
\begin{minipage}[b]{0.5\linewidth}
\centering
\includegraphics[keepaspectratio, scale=0.5]{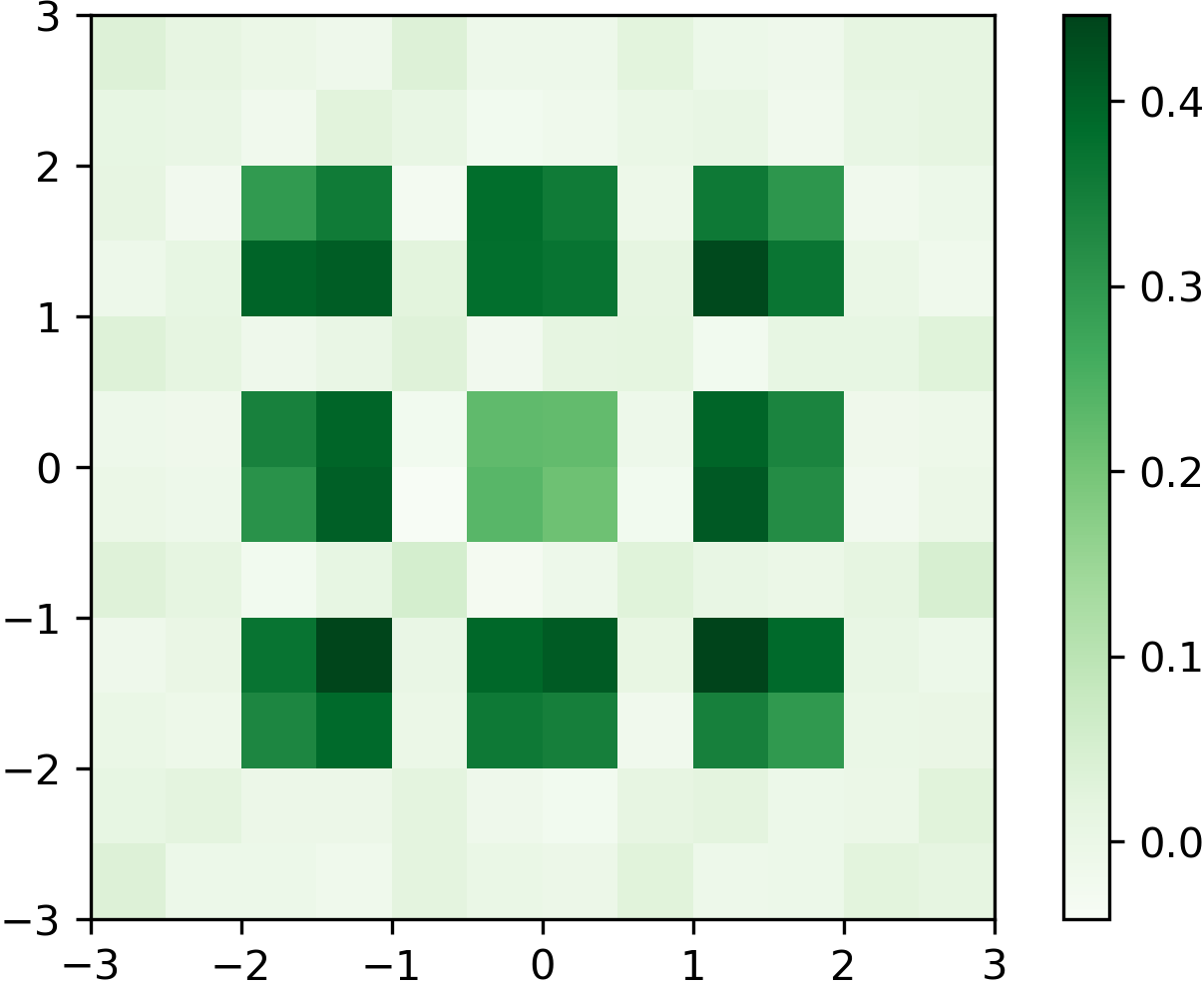}
\subcaption{KFL-update}
\end{minipage}
\end{tabular}
\begin{tabular}{c}
\hspace{2cm}
\begin{minipage}[b]{0.6\linewidth}
\centering
\includegraphics[keepaspectratio, scale=0.67]{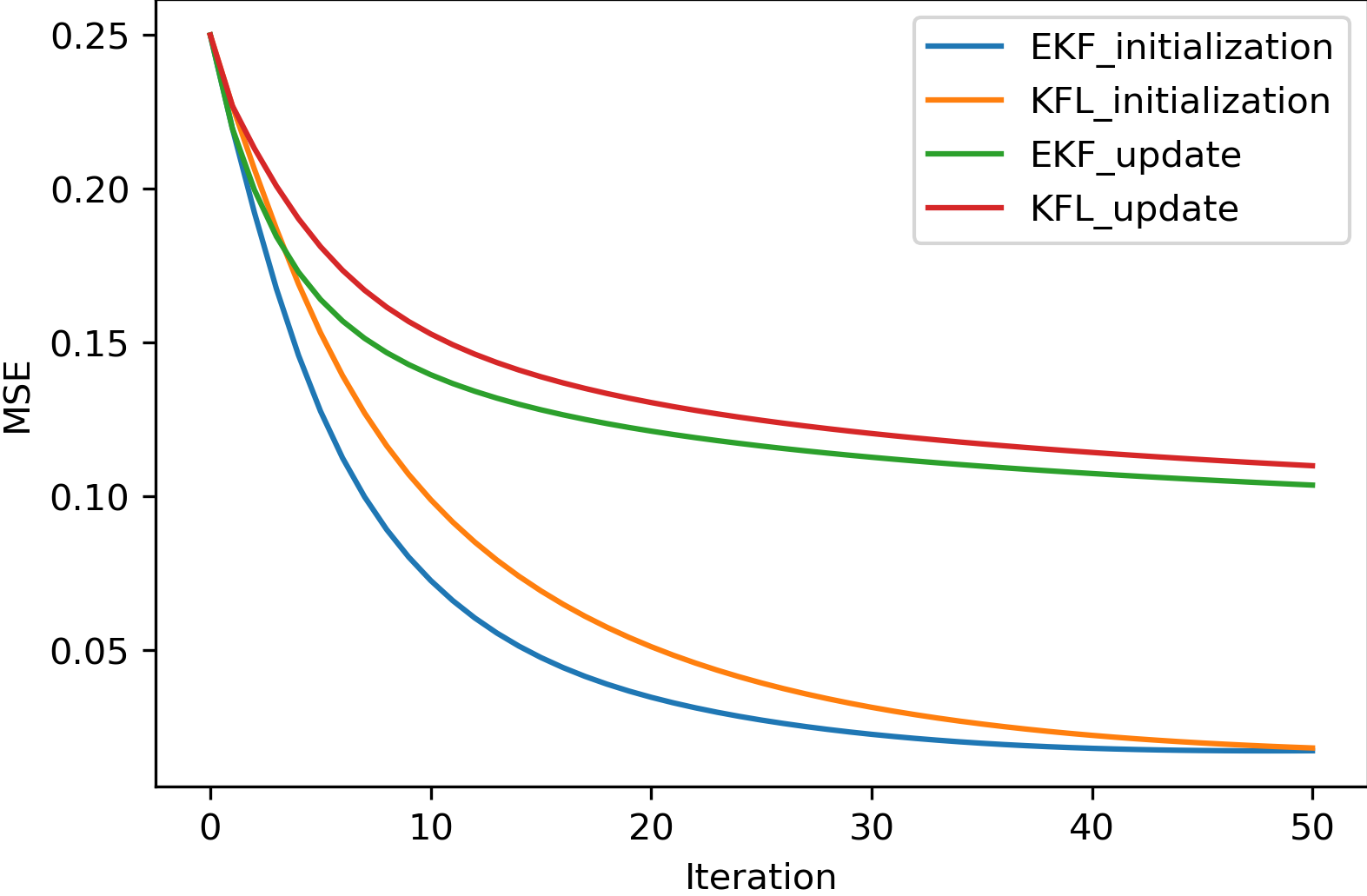}
\end{minipage}
\vspace{1cm}
\end{tabular}
\caption{Reconstruction for the noise-free case: $q_{1}^{true}$, $\alpha=10000$, $\sigma=0.5$}
\label{B2noise}
\end{figure}

\begin{figure}[h]
\begin{tabular}{c}
\hspace{-2cm}
\begin{minipage}[b]{0.6\linewidth}
\centering
\includegraphics[keepaspectratio, scale=0.6]{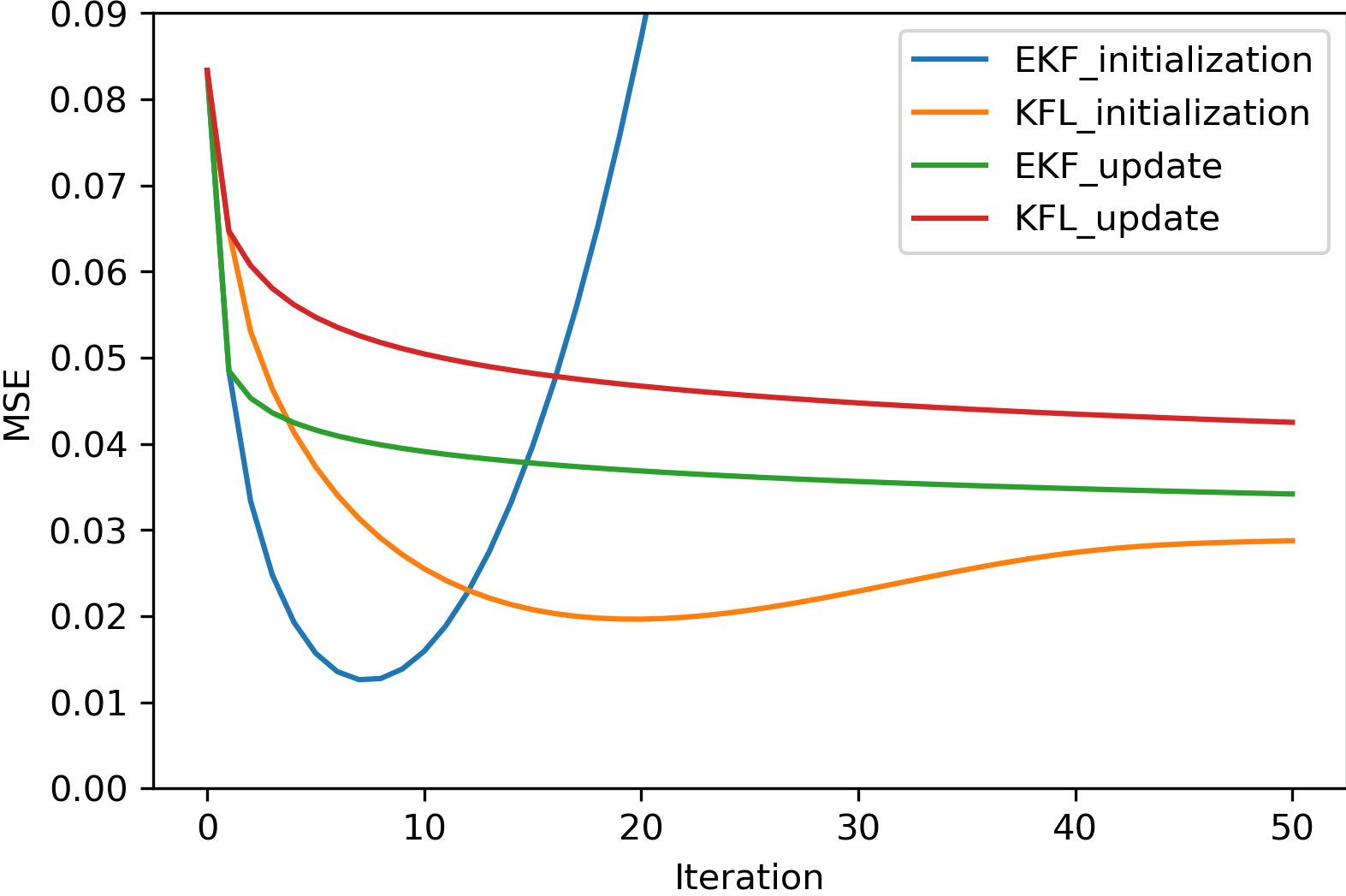}
\subcaption{$\sigma=0.6$}
\end{minipage}
\begin{minipage}[b]{0.6\linewidth}
\centering
\includegraphics[keepaspectratio, scale=0.6]{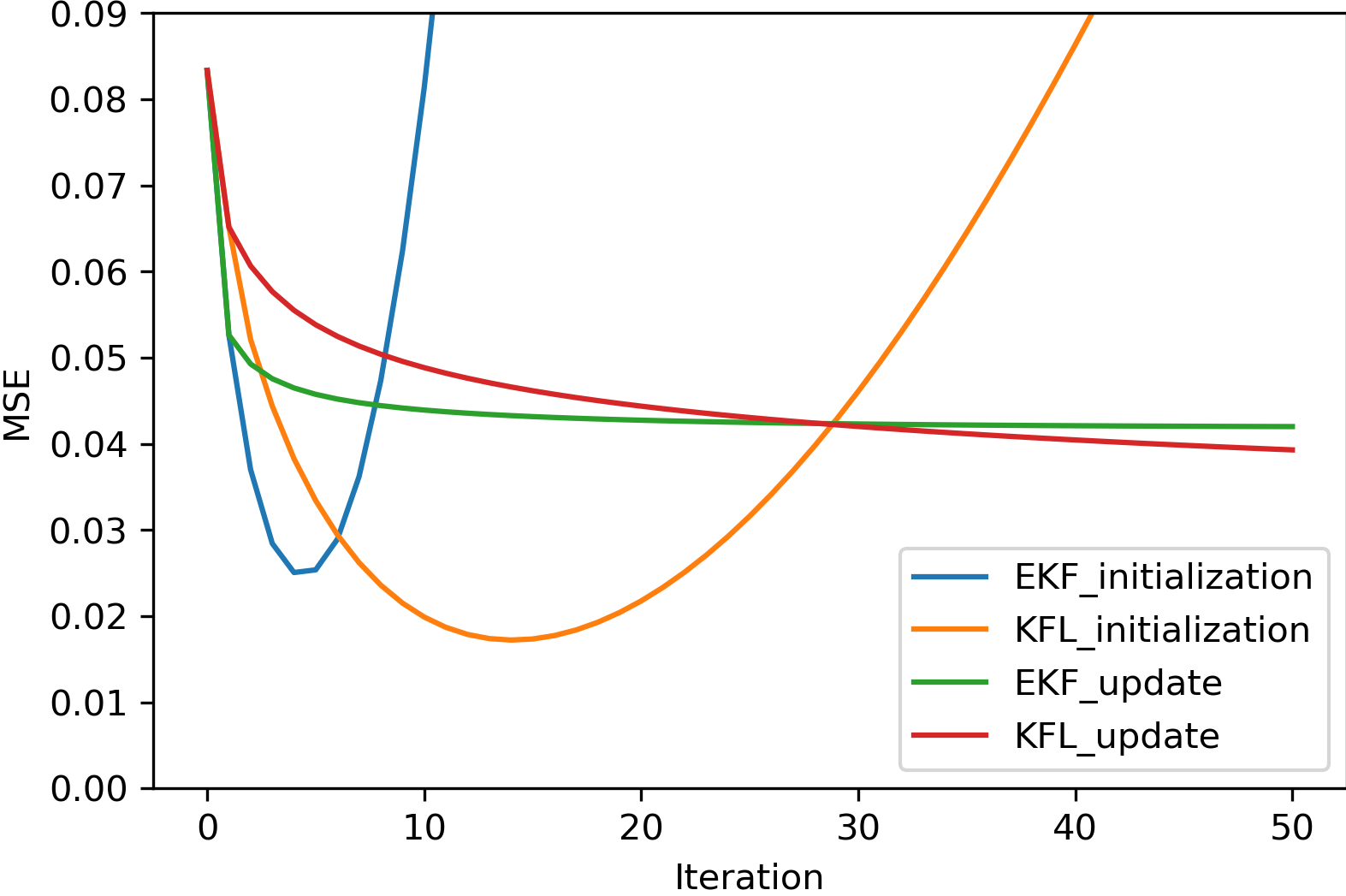}
\subcaption{$\sigma=0.9$}
\end{minipage}
\vspace{1cm}
\end{tabular}
\begin{tabular}{c}
\hspace{2cm}
\begin{minipage}[b]{0.6\linewidth}
\centering
\includegraphics[keepaspectratio, scale=0.6]{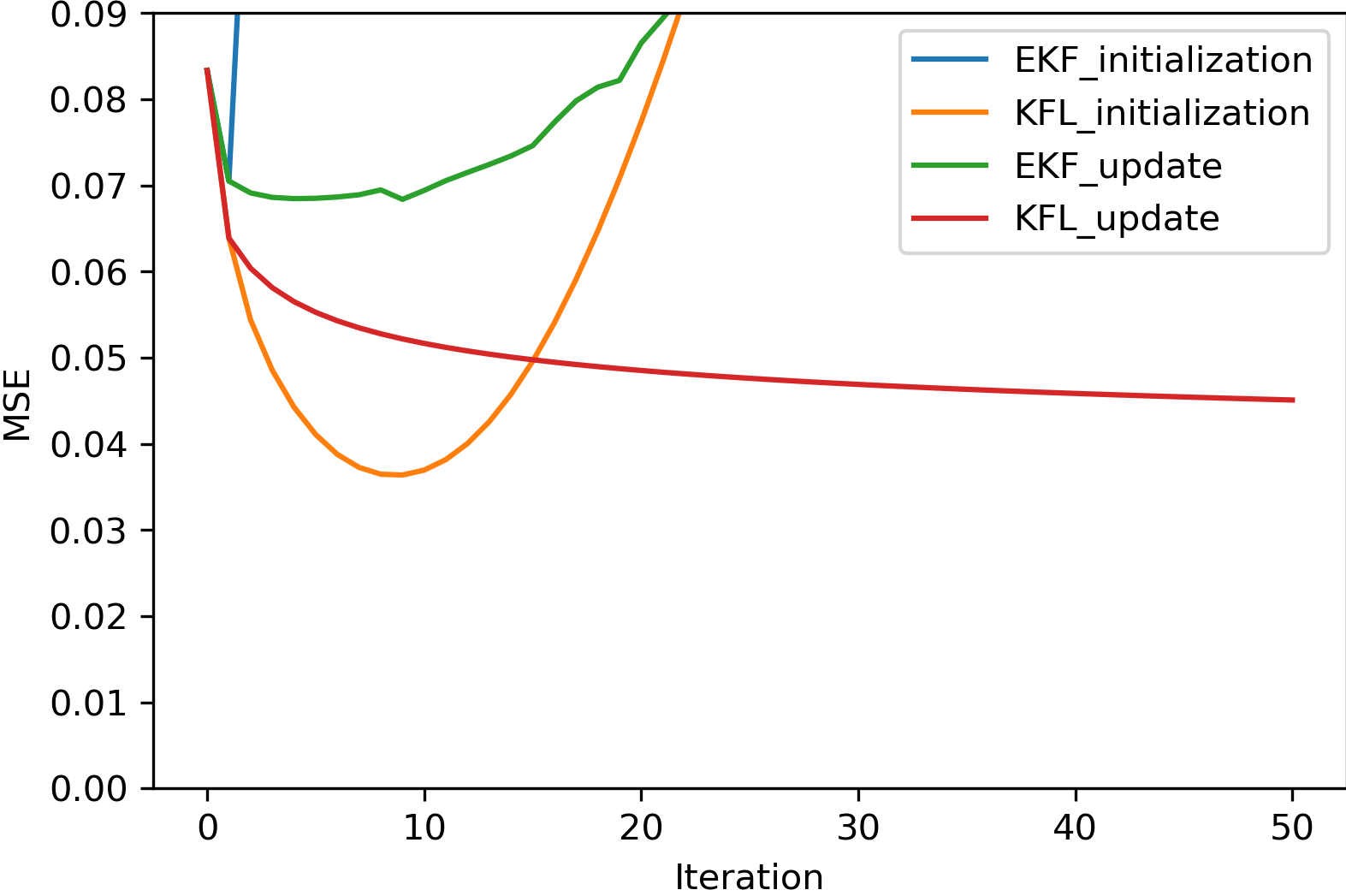}
\subcaption{$\sigma=1.2$}
\end{minipage}
\end{tabular}
\caption{MSE graph for different standard deviations $\sigma$: $q_{1}^{true}$, $\alpha=100$}
\label{Graph sigma}
\end{figure}

\begin{figure}[h]
\begin{tabular}{c}
\hspace{-2cm}
\begin{minipage}[b]{0.6\linewidth}
\centering
\includegraphics[keepaspectratio, scale=0.6]{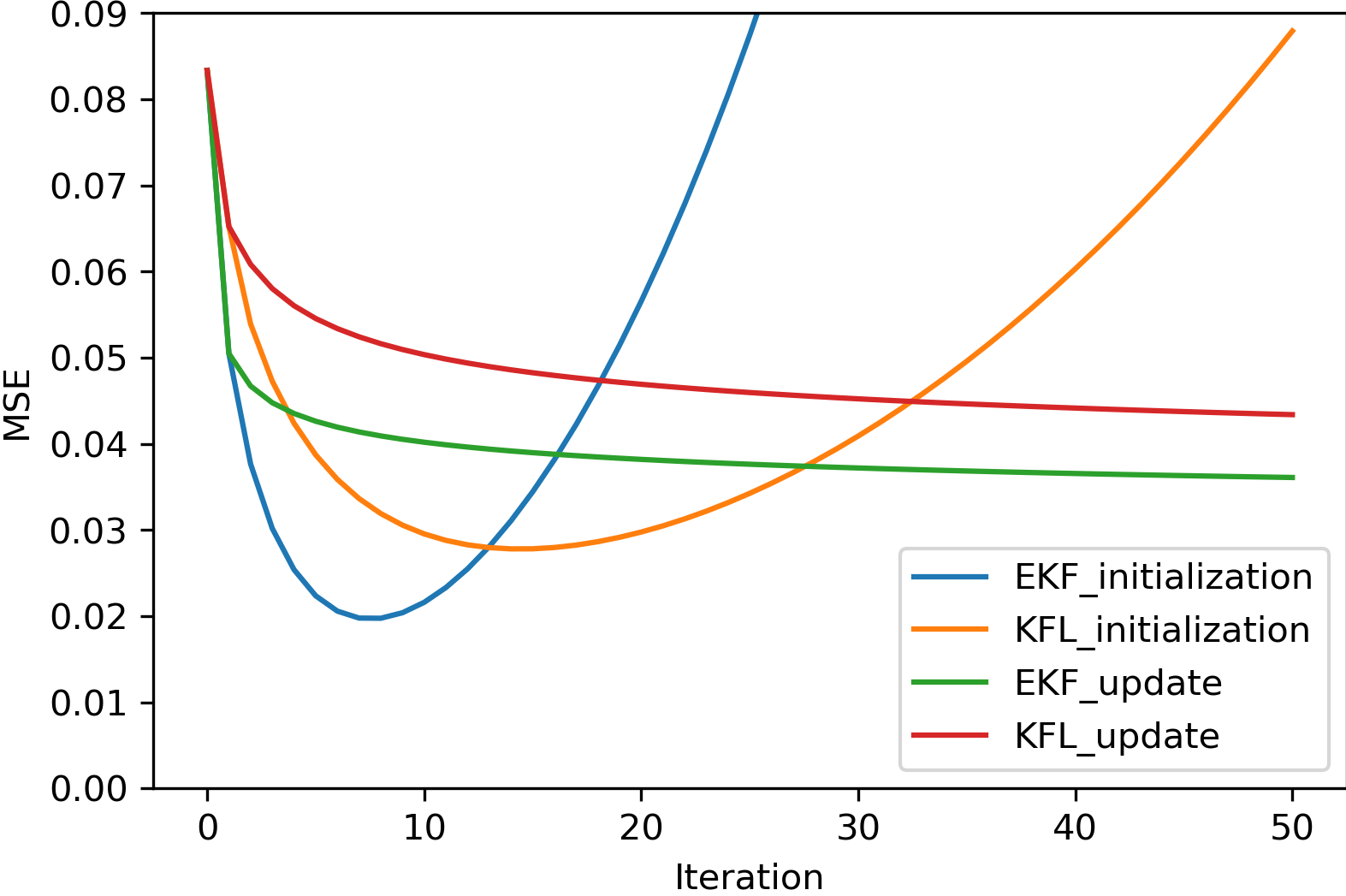}
\subcaption{$\alpha=300$}
\end{minipage}
\begin{minipage}[b]{0.6\linewidth}
\centering
\includegraphics[keepaspectratio, scale=0.6]{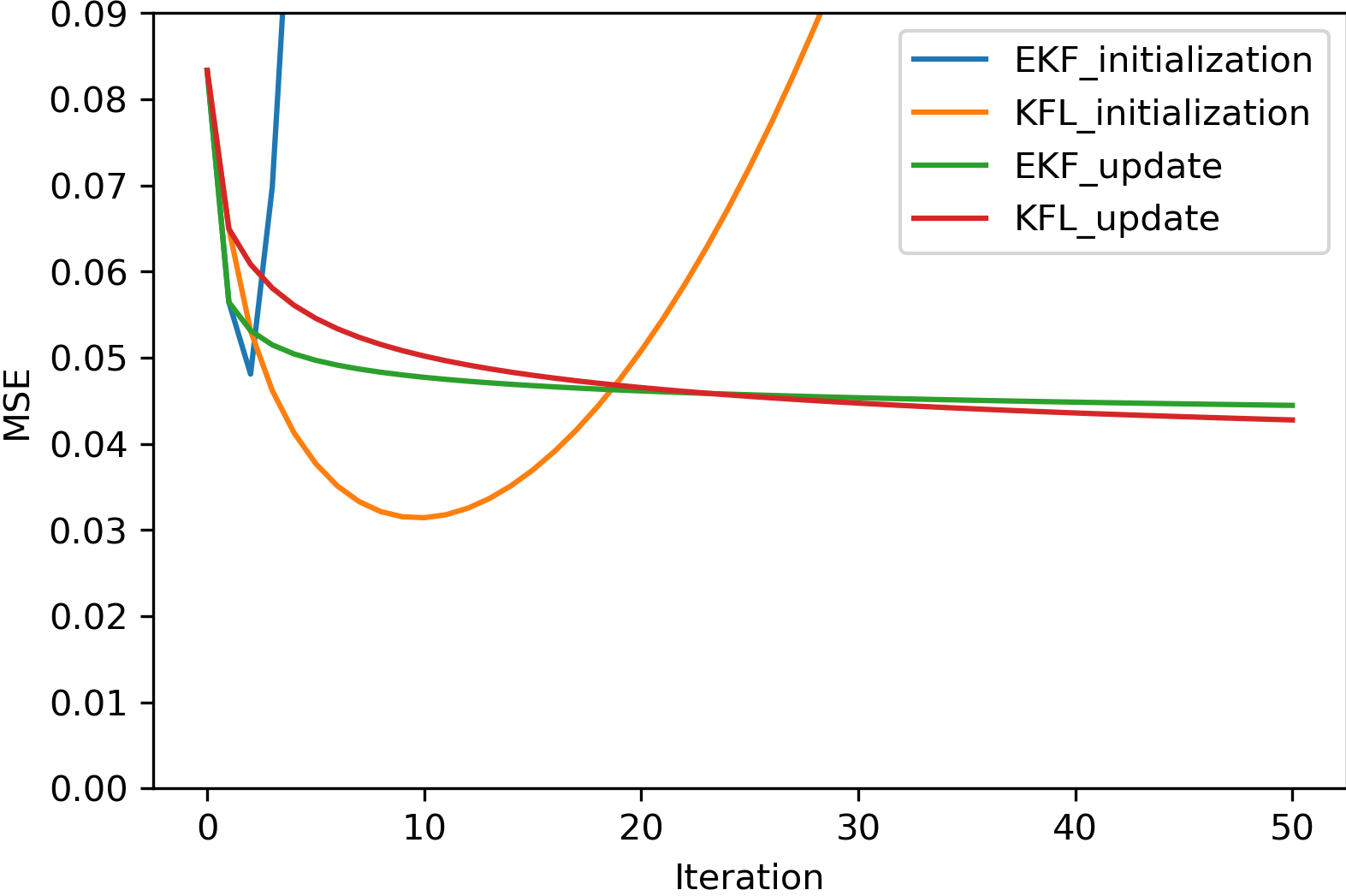}
\subcaption{$\alpha=100$}
\end{minipage}
\vspace{1cm}
\end{tabular}
\begin{tabular}{c}
\hspace{2cm}
\begin{minipage}[b]{0.6\linewidth}
\centering
\includegraphics[keepaspectratio, scale=0.6]{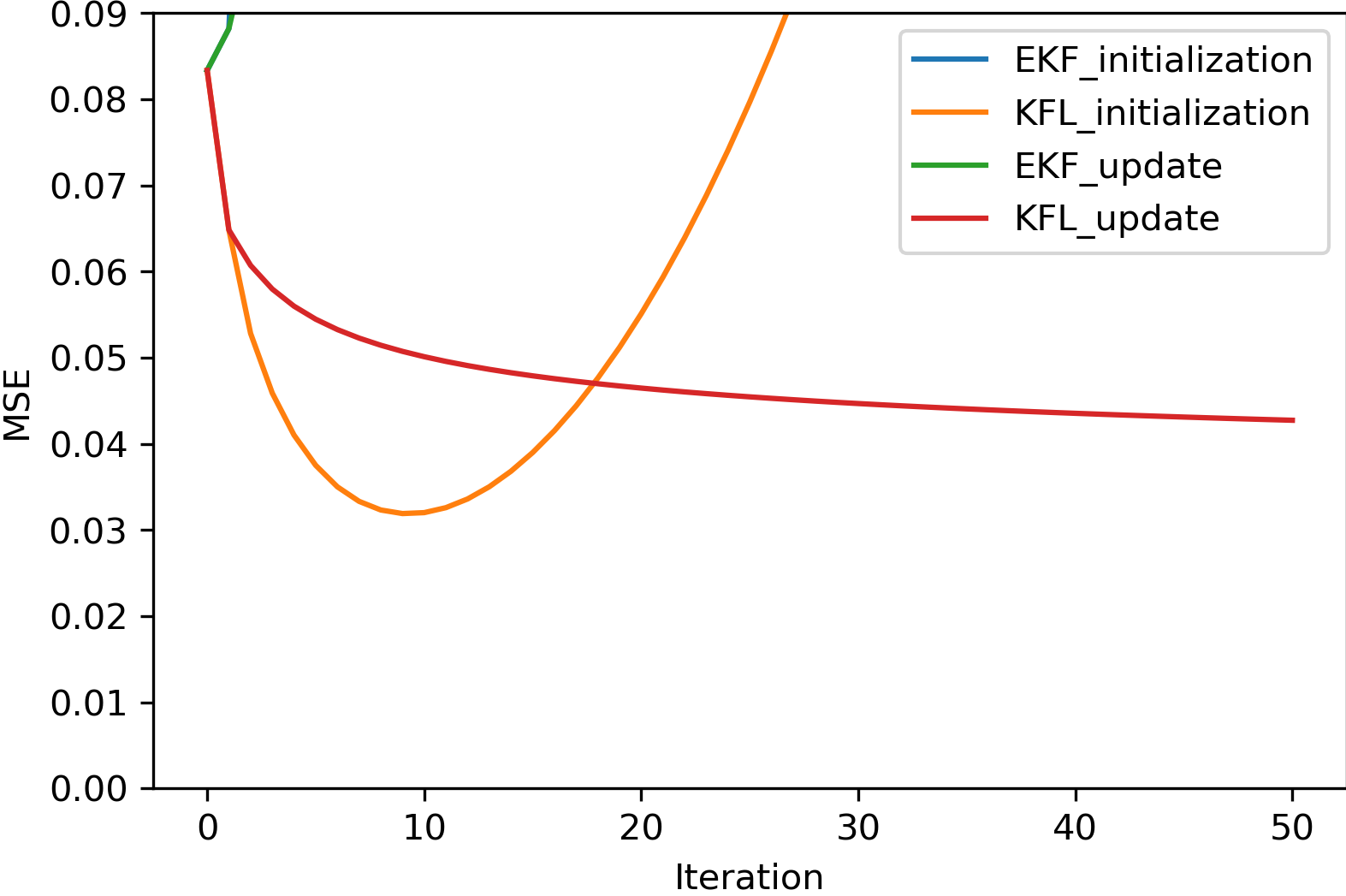}
\subcaption{$\alpha=50$}
\end{minipage}
\end{tabular}
\caption{MSE graph for different regularization parameters $\alpha$: $q_{1}^{true}$, $\sigma=1$}
\label{Graph alpha}
\end{figure}

\end{document}